\documentclass[11pt]{amsart}

\title[Non-separable mean field games]
{Existence theory for  non-separable mean field games in Sobolev spaces}
\author{David M. Ambrose}
\address{Drexel University, Department of Mathematics, Philadelphia, PA, USA}
\email{dma68@drexel.edu}

\newtheorem{theorem}{Theorem}
\newtheorem{lemma}[theorem]{Lemma}
\newtheorem{corollary}[theorem]{Corollary}
\newtheorem{remark}{Remark}

\begin{document}

\begin{abstract} The mean field games system is a coupled pair of nonlinear partial differential equations
arising in differential game theory, as a limit as the number of agents tends to infinity.
We prove existence and uniqueness of classical solutions for time-dependent mean field
games with Sobolev data.  Many works in the literature assume additive separability of the Hamiltonian,
as well as further structure such as convexity and monotonicity of the resulting components.
Problems arising in practice, however, may not have this separable structure; we therefore consider
the non-separable problem.  For our existence and uniqueness results, 
we introduce new smallness constraints which simultaneously consider the 
size of the time horizon,
the size of the data, and the strength of the coupling in the system.  
\end{abstract}

\maketitle

\section{Introduction}

Mean field games have been introduced in the mathematics literature by Lasry and Lions
as limits of problems from game theory, as the number of agents tends to infinity \cite{lasry1},
\cite{lasry2}, \cite{lasry3}.  From a control theory perspective, mean field games were also introduced
around the same time by Huang, Caines, and Malhame \cite{huang1}, \cite{huang2}.
The mean field games system of partial differential equations is the following coupled system for a
value function, $u,$ and a probability measure, $m:$
\begin{equation}\label{uEquation}
u_{t}+\Delta u +\mathcal{H}(t,x,m,Du)=0,
\end{equation}
\begin{equation}\label{mEquation}
m_{t}-\Delta m +\mathrm{div}\left(m\mathcal{H}_{p}(t,x,m,Du)\right)=0,
\end{equation}
with $x\in\mathbb{T}^{d}$ and $t\in[0,T],$ for some given $T>0.$  
The function $\mathcal{H}$ is known as the Hamiltonian, and we treat the local case, in which 
$\mathcal{H}:\mathbb{R}^{2d+2}\rightarrow\mathbb{R}$ is a function
of its arguments, as opposed to the nonlocal case in which it might involve an integral operator applied to the unknowns.
These equations are supplemented
with initial and terminal boundary conditions.  We primariy assume that the initial value of $m$ and the final 
value of $u$ are prescribed functions,
\begin{equation}\label{planningBC}
m(0,x)=m_{0}(x),\qquad u(T,x)=u_{T}(x).
\end{equation}
This is a special case of the more general payoff problem, which uses the following:
\begin{equation}\label{payoffBC}
m(0,x)=m_{0}(x),\qquad u(T,x)=G(x,m(T,x)),
\end{equation}
where $G$ is known as the payoff function.  While we will focus on \eqref{planningBC} for simplicity, 
we can treat the more general problem \eqref{payoffBC} as well,
and will discuss this in some detail in Section \ref{payoffExtension} below.

The author has previously adapted ideas from the work
of Duchon and Robert on vortex sheets in incompressible flow \cite{duchonRobert}
to prove existence of strong solutions for mean field games.
Duchon and Robert developed a Duhamel formula for the
vortex sheet which integrated both forward and backward in time, and found a contraction 
in function spaces based on the Wiener algebra, proving the existence of small spatially analytic
solutions.  The ideas of Duchon and Robert have been extended to finite time horizons and the spatially
periodic setting by Milgrom and the author \cite{milgromAmbrose}.
All of these features are thus also characteristics
of the author's work \cite{ambroseMFG} on mean field games.  This was extended somewhat
in \cite{ambroseMFG2}, in which non-separable Hamiltonians were treated, and a result
in the case of weak coupling, making use of the implicit function theorem, was also given.

Other authors have proved existence theorems for mean field games, focusing much attention on the 
case of separable mean field games.  The assumption of separability is that the Hamiltonian, $\mathcal{H},$
separates additively as $\mathcal{H}(t,x,m,Du)=H(t,x,Du)+F(t,x,m).$  This $H$ is then also known as the Hamiltonian,
and the function $F$ is known as the coupling (for if one were to take $F=0,$ then the system decouples).  
The separability assumption, as well as further structural assumptions such as convexity of $H$ and monotonicity of 
$F,$ allow certain mathematical methods to be brought to bear on the problems (i.e., use of convex optimization and 
montonoicity methods, as well as techniques of optimal transportation).
Porretta has proved in the separable case, using such techniques, the existence of weak solutions 
\cite{porrettaCRAS}, \cite{porretta1}, \cite{porretta2}.
Results in this vein for strong solutions
are by Gomes, Pimentel, and Sanchez-Morgado in the case of superquadratic and subquadratic
Hamiltonians \cite{gomesSuper}, \cite{gomesSub}, and by 
Gomes and Pimentel for the case of logarithmic coupling \cite{gomesLog}.

Although the separable case does have a number of sophisticated mathematical techniques available for existence
theory, unfortunately, problems actually arising from game theory and economics do not tend to have this separable
structure \cite{mollNotes}.  Therefore a study of existence theory not relying on this structure is essential. 
A particular example of a nonseparable mean field game arising in practice is a model of household
wealth \cite{pdemacro}, \cite{achdouMacro2}.  The author has proved the first existence theorem for the time-dependent version
of this model for household wealth \cite{ambroseHousehold}.  The techniques of the present paper are related to those used in
\cite{ambroseHousehold}, but we treat a general class of Hamiltonians here rather than the one arising from the specific application.
 
As mentioned above, the author has previously made one study of mean field games with non-separable Hamiltonians 
\cite{ambroseMFG2}, and this work contained two
different results.   The first of these placed a smallness condition on the data, and the other considered a small parameter in front of
the Hamiltonian.  There are a few other existence results in the literature for mean field games with non-separable Hamiltonians.
Cirant, Gianni, and Mannucci have
proven an existence theorem for non-separable mean field games in Sobolev spaces, under a smallness condition on 
the time horizon \cite{cirantNew}; they also place a restriction on the initial distribution of agents, requiring that this be bounded away from
zero (we have no such restriction).  In \cite{cirantNew} the boundary conditions \eqref{payoffBC} are used, with the assumption that
the payoff function $G$ is smoothing.  We will focus primarily on the boundary conditions \eqref{planningBC} in the present work;
since in this case we consider $u$ to be more regular than $m,$ this falls into their paradigm of having $G$ be smoothing.  We do, however,
go beyond this, proving a theorem in Section \ref{discussionSection} in which $G$ is not assumed to be smoothing.
The only other proofs in the literature of existence for non-separable mean field games of which the author is aware
are for a particular form of Hamiltonian related to modeling problems with congestion 
\cite{gomesCongestion}; in this work, Gomes and Voskanyan made a smallness assumption on $T,$ the length
of the time horizon, and still do make structural assumptions such as monotonicity on part of the Hamiltonian.  We also 
address problems with congestion in Section \ref{discussionSection}.

In the present work, we introduce a unified smallness condition which considers at once
the size of the time horizon, the coupling parameter introduced by the author in \cite{ambroseMFG2}, and in some cases,
the size of the data.  In addition to unifying the smallness constraints, a benefit of the present work as compared to
\cite{ambroseMFG2} is the setting of more customary Sobolev spaces as opposed to the spaces based on the Wiener
algebra used previously.
We also prove a uniqueness theorem, and as in the case of our existence theorem, a smallness condition must be satisfied.
This smallness condition again considers at once a parameter which we describe as measuring the coupling in the system
(unrelated to the concept of coupling in the separable case), the size of the time horizon, and in some cases, the size
of the initial data.  Such a smallness constraint is perhaps not just a feature of our proof, but may be more fundamental:
Bardi and Fischer have recently given an example in mean field games of non-unique solutions, in the case of large time
horizon \cite{bardiFischer}.  Cirant and Tonon have also given an example of non-uniqueness \cite{cirantTonon}.
While their settings may not be exactly the same as ours (we study the problem on the torus, 
and their construction uses the domain as the real line in a fundamental way), it is strongly suggestive that constraints such
as those we impose are not in general avoidable.  
Bardi and Cirant have a related uniqueness theorem, for separable
mean field games with some smallness constraints \cite{bardiCirant}.

The plan of the paper is as follows.  Immediately below, in Section \ref{preliminaries}, we give some elementary 
definitions and results on Sobolev spaces.  In Section \ref{formulation}, we reformulate the problem slightly and introduce
an approximating sequence for solutions.  In Section \ref{uniformEstimates}, 
we prove our first main theorem (stated at the end of the
section as Theorem \ref{existenceTheorem}), 
that under our smallness assumption, the approximating sequence converges to a
solution of the mean field games system.  We next treat uniqueness of solutions in Section \ref{uniquenessSection}, 
stating our second main
theorem, Theorem \ref{uniquenessTheorem}, at the end of the section.  
We close with some discussion in Section \ref{discussionSection}.  Included in this discussion section is a proof 
of existence of solutions for mean field games with nonsmoothing payoff functions, and also as a corollary of our main theorem,
a proof of existence of solutions for mean field games with congestion

\subsection{Function spaces and preliminaries}\label{preliminaries}

We will make repeated use of Young's Inequality:  for any $a\geq0,$ for any $b\geq 0,$ and for any $\sigma>0,$
we have
\begin{equation}\label{young}
ab\leq \frac{a^{2}}{2\sigma}+\frac{\sigma b^{2}}{2}.
\end{equation}
To be very definite, we say that we let $\mathbb{N}=\{0,1,2,\ldots\}$ be the natural numbers,
including zero.

We now define our function spaces and norms.  The $d$-dimensional torus is the set $[0,2\pi]^{d}$ with
periodic boundary conditions.  The Fourier transform of a function $f$ may be denoted either as $\mathcal{F}f(k)$
or $\hat{f}(k),$ with Fourier variable $k\in\mathbb{Z}^{d}.$
Of course, the Sobolev space $H^{0}(\mathbb{T}^{d})$ is equal to $L^{2}(\mathbb{T}^{d}),$ with the same norm.
 We need multi-index notation for derivatives with respect to the $x$ variables.  
We will use $\alpha\in\mathbb{N}^{d}$ for this purpose.  Thus, given such an $\alpha,$ we will have 
$\partial^{\alpha}=\partial_{x_{1}}^{\alpha_{1}}\cdots\partial_{x_{d}}^{\alpha_{d}}.$  The order of
$\alpha$ is $|\alpha|=\displaystyle\sum_{\ell=1}^{d}\alpha_{\ell}.$  For $s\in\mathbb{N},$ with $s>0,$
the Sobolev space of order $s$ is the set of functions
\begin{equation}\nonumber
H^{s}(\mathbb{T}^{d})=\left\{f\in L^{2}(\mathbb{T}^{d}): \|f\|_{s}<\infty\right\},
\end{equation}
where the norm is defined by
\begin{equation}\nonumber
\|f\|_{s}^{2}=\sum_{0\leq|\alpha|\leq s}\|\partial^{\alpha}f\|_{0}^{2}.
\end{equation}
Here, as is usual, the notation $\|\cdot\|_{0}=\|\cdot\|_{L^{2}}$ is used.
This definition is equivalent to any other usual definition of Sobolev spaces with index in the natural numbers.
We need an elementary interpolation lemma, which we now state.

\begin{lemma}\label{elementaryInterpolation}
Let $m$ and $s$ be real numbers such that $0<m<s.$  There exists $c>0$ such that for all $f\in H^{s},$
\begin{equation}\nonumber
\|f\|_{m}\leq c\|f\|_{s}^{m/s}\|f\|_{0}^{1-m/s}.
\end{equation}
\end{lemma}
We do not include a proof of Lemma \ref{elementaryInterpolation}; the proof can be found many places, one of which
is \cite{ambroseThesis}.
We also need an elementary lemma about products in Sobolev spaces; this is part of Lemma 3.4 of \cite{majdaBertozzi},
and the proof can be found there.

\begin{lemma}\label{leadingTermsSobolev} 
Let $m\in\mathbb{N}.$  There exists $c>0$ such that for all $f\in L^{\infty}\cap H^{m}$ and for all $g\in L^{\infty}\cap H^{m},$
\begin{equation}\nonumber
\sum_{0\leq|\alpha|\leq m}\|\partial^{\alpha}(fg)-f\partial^{\alpha}g\|_{L^{2}}
\leq c\left(|Df|_{\infty}\|D^{m-1}g\|_{L^{2}}+\|D^{m}f\|_{L^{2}}|g|_{\infty}\right).
\end{equation}
\end{lemma}

\section{Formulation and The Approximating Sequence}\label{formulation}

As we have said, we will track three effects in our existence theorem: the size of the time horizon 
(i.e., the magnitude of the value $T$),
the size of the initial data, and the strength of the coupling between the two equations of the mean field games
system.  We now explain what we mean by this third item.  We introduce a slight modification of the system
\eqref{uEquation}, \eqref{mEquation}:
\begin{equation}\nonumber
u_{t}+\Delta u+\varepsilon\mathcal{H}(t,x,m,Du)=0,
\end{equation}
\begin{equation}\nonumber
m_{t}-\Delta m+\varepsilon\mathrm{div}(m\mathcal{H}_{p}(t,x,m,Du))=0,
\end{equation}
for some $\varepsilon\in\mathbb{R}.$
Obviously if $\varepsilon=1,$ this is exactly the system \eqref{uEquation}, \eqref{mEquation}.
If instead $\varepsilon=0,$ then the system decouples -- one could solve the linear heat equation for $m$
and then the other linear heat equation for $u.$  We may call the case of small values of $\varepsilon$ the
case of weak coupling of the system, and one of the theorems of \cite{ambroseMFG2} was in the case
of weak coupling.  We will perform our existence theory for the system with $\varepsilon$ included as 
a parameter.

We let $\bar{m}=\displaystyle\frac{1}{\mathrm{vol}(\mathbb{T}^{d})},$ which is the average value of
$m$ (since $m$ is a probability distribution).  It is convenient to introduce $\mu=m-\bar{m}.$
We also subtract the mean from $u,$ since inspection of the right-hand sides of the evolution equations
indicates that the mean of $u$ does not influence the evolution.  We introduce a projection operator, $P,$
which removes the mean of a periodic function (so, we could have said before that $\mu=Pm$), and
denote $w=Pu.$  Note that then $Dw=Du.$

We now introduce the system of equations satisfied by $(w,\mu),$
first giving notation for the Hamiltonian in terms of $(w,\mu):$
\begin{equation}\nonumber
\Theta(t,x,\mu,Dw)=\mathcal{H}(t,x,m,Du).
\end{equation}
We then have the $(w,\mu)$ system:
\begin{equation}\label{uMu1}
w_{t}+\Delta w+\varepsilon P\Theta(t,x,\mu,Dw)=0,
\end{equation}
\begin{equation}\label{uMu2}
\mu_{t}-\Delta\mu+\varepsilon\mathrm{div}(\mu\Theta_{p}(t,x,\mu,Dw))
+\varepsilon\bar{m}\mathrm{div}(\Theta_{p}(t,x,\mu,Dw))=0.
\end{equation}
Of course, we have initial data for $\mu:$
\begin{equation}\nonumber
\mu(0,x)=\mu_{0}(x):=m(0,x)-\bar{m}.
\end{equation}
We will discuss the data for $w$ soon below.

\begin{remark} We will be proving an existence theorem with a smallness condition.
The reason for subtracting $\bar{m}$ from $m$ to form $\mu$ and for replacing $u$ with $w$ 
is to clarify this smallness condition.
Taking $m$ arbitrarily small is not compatible with the fact that $m$ should be a probability
measure.  Furthermore, since the mean of $u$ is not relevant for the right-hand sides of the equation, 
requiring the mean of $u$ to be small would be artificial.
Instead, the smallness condition will include information about the
initial size of $\mu,$ and about the size of the data for $w;$ as far as $m$ goes, then, 
we will be measuring how far $m$ is from a uniform distribution.  
\end{remark}

It is convenient to introduce a regularization operator, which will be useful as we construct solutions.
Let $\delta>0$ be given.  We let $\mathbb{P}_{\delta}$ be the Fourier multiplier operator which projects onto 
Fourier modes with wavenumber
at most $1/\delta:$
\begin{equation}\nonumber
\mathcal{F}\left(\mathbb{P}_{\delta}f\right)(k)=\left\{\begin{array}{ll}\mathcal{F}f(k),&|k|\leq 1/\delta,\\
0,&|k|>1/\delta.\end{array}\right.
\end{equation}
We may use the convention $\mathbb{P}_{0}=I,$ where this signifies the identity operator.

We then set up an iterative approximation scheme, which will depend slightly on the choice of
boundary conditions.  In either case, we initialize in the same way, and we solve for $\mu$ in the same way.
Define $\mu^{0}=0$ and $w^{0}=0.$
Given $(w^{n},\mu^{n}),$ we define $\mu^{n+1}$ to be the unique solution of
the initial value problem for the following forced, linear heat equation:
\begin{equation}\label{muN}
\mu^{n+1}_{t}-\Delta\mu^{n+1}+\varepsilon\mathrm{div}(\mu^{n}\Theta_{p}(t,x,\mu^{n},Dw^{n}))
+\varepsilon\bar{m}\mathrm{div}(\Theta_{p}(t,x,\mu^{n},Dw^{n}))=0,
\end{equation}
\begin{equation}\label{muN+1BC}
\mu^{n+1}(0,x)=\mathbb{P}_{\delta}\mu_{0}(x).
\end{equation}

We define $w^{n+1}$ to be the unique solution of the following forced, linear heat equation:
\begin{equation}\label{uN}
w^{n+1}_{t}+\Delta w^{n+1}+\varepsilon P\Theta(t,x,\mu^{n},Dw^{n})=0,
\end{equation}
\begin{equation}\label{uN+1BC}
w^{n+1}(T,x)=\mathbb{P}_{\delta}w_{T}(x),
\end{equation}
where $w_{T}=Pu_{T}.$
We may be completely explicit as to what these solutions $(w^{n+1},\mu^{n+1})$ are, by using Duhamel's
formula; for $w$ we have
\begin{equation}\label{uN+1ByDuhamel}
w^{n+1}(t,\cdot)=e^{\Delta(T-t)}\mathbb{P}_{\delta}w_{T}-\varepsilon P\int_{t}^{T}e^{\Delta(s-t)}
\Theta(s,\cdot,\mu^{n}(s,\cdot),Dw^{n}(s,\cdot))\ ds.
\end{equation}
For $\mu,$ we instead integrate forward in time, finding
\begin{multline}\label{muN+1ByDuhamel}
\mu^{n+1}(t,\cdot)=e^{\Delta t}\mathbb{P}_{\delta}\mu_{0}-\varepsilon\int_{0}^{t}e^{\Delta(t-s)}
\mathrm{div}\left(\mu^{n}(s,\cdot)\Theta_{p}(s,\cdot,\mu^{n}(s,\cdot),Dw^{n}(s,\cdot))\right)\ ds
\\
-\varepsilon\bar{m}\int_{0}^{t}e^{\Delta(t-s)}\mathrm{div}\left(
\Theta_{p}(s,\cdot,\mu^{n}(s,\cdot),Dw^{n}(s,\cdot))\right)\ ds.
\end{multline}

\begin{remark}
Because of the presence of the projection $\mathbb{P}_{\delta},$ the initial and terminal data for $\mu^{n}$ and $w^{n},$
respectively, for all $n,$ is infinitely smooth.  Furthermore, $\mu^{n}$ and $w^{n}$ satisfy linear heat equations.  It is trivial
to show by induction, then, that for all $n,$ the solutions given by \eqref{uN+1ByDuhamel} and 
\eqref{muN+1ByDuhamel} are infinitely smooth at each
time $t\in[0,T],$ at least if the Hamiltonian is $C^{\infty}$ (the regularity of the iterates is otherwise only limited by
the regularity of $\mathcal{H}$).  
This fact helps to justify the estimates to be carried out in Section \ref{uniformEstimates} below.
\end{remark}

\section{Uniform Estimates and Existence of Solutions}\label{uniformEstimates}

Having defined a sequence of approximate solutions $(w^{n},\mu^{n})$ in Section \ref{formulation},
we will now work towards passing to the limit as $n$ goes to infinity.  In the present section, we will state
assumptions on the Hamiltonian which will allow us to make estimates uniform in $n.$

We introduce now some further multi-index notation.  We will
use this for denoting derivatives of $\Theta.$ 
Consider $\Theta=\Theta(t,x_{1},\ldots,x_{d},q,p_{1},\ldots,p_{d}).$ 
A multi-index $\beta$ is an element of $\mathbb{N}^{2d+1};$ the first $d$ positions 
correspond to the spatial variables $x_{1}, x_{2}, \ldots x_{d},$ the $(d+1)^{\mathrm{st}}$ position corresponds to the 
variable $q$ (which is a placeholder for $\mu$), and the final $d$ positions correspond to the $p$ 
variables.  Derivatives with respect to such a multi-index $\beta$ are denoted $\partial^{\beta},$
and the order of $\beta$ is $|\beta|=\displaystyle\sum_{\ell=1}^{2d+1}\beta_{\ell},$
as is usual.  We make the following assumption on $\mathcal{H}:$\\

\noindent{\bf(H1)} The function $\mathcal{H}$ is such that there exists a non-decreasing function
$\tilde{F}:[0,\infty)\rightarrow[0,\infty)$ such that for all $\beta\in\mathbb{N}^{2d+1}$ with $|\beta|\leq s+2,$
\begin{equation}\nonumber
\left|\partial^{\beta}\Theta(\cdot,\cdot,\nu,Dy)\right|_{\infty}\leq 
\tilde{F}\left(|\nu|_{\infty}+|Dy|_{\infty}\right).
\end{equation}

We use Sobolev embedding to replace $\tilde{F}$ with the closely related $F,$ which is also a non-decreasing function and
which instead satisfies
\begin{equation}\label{thetaInfinityBound}
\left|\partial^{\beta}\Theta(\cdot,\cdot,\nu,Dy)\right|_{\infty}\leq 
F\left(\|\nu\|_{\left\lceil\frac{d+1}{2}\right\rceil}^{2}+\|Dy\|_{\left\lceil\frac{d+1}{2}\right\rceil}^{2}\right),
\end{equation}
for all $\beta$ as above.
Based on this assumption, we may conclude a useful lemma.

\begin{lemma}\label{boundsOnTheta}
Assume {\bf(H1)} is satisfied.  
There exist constants $\bar{c}>0$ and $\bar{C}>0$ such that for multi-indices $\beta$ (as defined at the beginning of this
section) and for
multi-indices $\alpha$ (as defined in Section \ref{preliminaries}), 
\begin{multline}\nonumber
\sum_{|\beta|\leq 2}\left(\sum_{|\alpha|\leq s-1}\left\|\partial^{\alpha}\left((\partial^{\beta}\Theta)(t,\cdot,\mu,Dw)\right)\right\|_{L^{2}(\mathbb{T}^{d})}\right)
\\
\leq\bar{c}F\left(\|\mu\|_{\left\lceil\frac{d+1}{2}\right\rceil}^{2}+\|Dw\|_{\left\lceil\frac{d+1}{2}\right\rceil}^{2}\right)
\left(1+\|\mu\|_{s-1}+\|w\|_{s}\right)^{s-1},
\end{multline}
and furthermore,
\begin{multline}\label{usefulBoundOnTheta}
\sum_{|\beta|\leq 2}\left\|(\partial^{\beta}\Theta)(t,\cdot,\mu,Du)\right\|_{s-1}
\\
\leq\bar{C}F\left(\|\mu\|_{\left\lceil\frac{d+1}{2}\right\rceil}^{2}+\|Dw\|_{\left\lceil\frac{d+1}{2}\right\rceil}^{2}\right)
\left(1+\|\mu\|_{s-1}+\|w\|_{s}\right)^{s-1}.
\end{multline}
\end{lemma}

It is helpful to expand the divergence which appears in the evolution equation for $\mu^{n+1}.$
We have the following formula:
\begin{multline}\label{expandedDivergence}
\mathrm{div}\left(\Theta_{p}(t,x,\mu^{n},Dw^{n})\right)=\sum_{i=1}^{d}
\Theta_{x_{i}p_{i}}(t,x,\mu^{n},Dw^{n})
\\
+\sum_{i=1}^{d}\left[\left(\Theta_{qp_{i}}(t,x,\mu^{n},Dw^{n})\right)(\partial_{x_{i}}\mu^{n})
\right]
+\sum_{i=1}^{d}\sum_{j=1}^{d}
\left[
(\Theta_{p_{i}p_{j}}(t,x,\mu^{n},Dw^{n}))(\partial^{2}_{x_{i}x_{j}}w^{n})
\right].
\end{multline}

\begin{remark}There are clearly three kinds of terms on the right-hand side of \eqref{expandedDivergence}.
We will be making energy estimates for $(w^{m},\mu^{m})\in H^{s}\times H^{s-1},$ for all $m.$  
The first kind of term involves no derivatives
of $\mu^{n}$ and first derivatives of $w^{n},$ and may be treated routinely in the estimates.  The second kind of terms
involve first derivatives on each of $\mu^{n}$ and $w^{n}.$  The first derivatives on $w^{n}$ cause no problems because of the
choice of function space.  The first derivatives on $\mu^{n}$ indicate that these are transport terms, which could typically
be treated in the energy estimate by integration by parts.  However, because of our iterative scheme, the necessary
structure for integration by parts is not present.  We will instead bound these terms using the available parabolic smoothing.
The third kind of term on the right-hand side of \eqref{expandedDivergence} involves no derivatives on $\mu^{n}$ and
second derivatives on $w^{n};$ these terms will also be bounded by taking advantage of parabolic smoothing.
\end{remark}

Similarly to the above, we apply $\partial_{x_{j}}$ to \eqref{uN}:
\begin{multline}\label{uN+1xJEvolution}
\partial_{x_{j}}w^{n+1}_{t}=-\Delta\partial_{x_{j}}w^{n+1}-\varepsilon \Theta_{x_{j}}(\cdot,\cdot,\mu^{n},Dw^{n})\\
-\varepsilon \left(\Theta_{q}(\cdot,\cdot,\mu^{n},Dw^{n})\right)\mu^{n}_{x_{j}}
-\varepsilon \sum_{i=1}^{d}\left[
\left(\Theta_{p_{i}}(\cdot,\cdot,\mu^{n},Dw^{n})\right)
\partial^{2}_{x_{i}x_{j}}w^{n}
\right].
\end{multline}
Notice that we have dropped the operator $P,$ since $\partial_{x_{j}}P=\partial_{x_{j}}.$

We now are ready to begin proving our existence theorem.  We will be making energy estimates for the unknowns,
and to get the estimates to yield control over norms of the unknowns, we will, as we have said, make a smallness 
assumption.  Unfortunately it would be quite complicated to state the smallness assumption before the estimates have
been carried out since it involves a constant which arises in the course of making the energy estimates.  Rather than
attempt to state this condition now, we will begin carrying out the energy estimates and make the assumption {\bf(H2)}
at the appropriate time below.  We now make an 
statement of the theorem to be proved, and will restate it again afterwards to be more specific about the smallness 
condition.
\begin{theorem}
Let $T>0$ and $\varepsilon>0$ be given.  
Let $s\geq \left\lceil\frac{d+5}{2}\right\rceil$ and let 
$\mu_{0}\in H^{s-1}(\mathbb{T}^{d})$ be such that $\bar{m}+\mu_{0}$ is a 
probability measure.  Let $u_{T}\in H^{s}(\mathbb{T}^{d})$ be given.  Assume that the condition
{\bf(H1)} is satisfied.  If the product $\varepsilon T F(8(\|\mu_{0}\|^{2}_{s-1}+\|Dw_{T}\|^{2}_{s-1}))$ is sufficiently small, 
then there exists $\mu\in L^{\infty}([0,T];H^{s-1})\cap L^{2}([0,T];H^{s})$ and
there exists $u\in L^{\infty}([0,T];H^{s})\cap L^{2}([0,T];H^{s+1})$ such that $\bar{m}+\mu$ is a probability measure for all
$t\in[0,T],$ and such that $(u,\bar{m}+\mu)$ is a classical solution of 
\eqref{uEquation}, \eqref{mEquation}, \eqref{planningBC}.
Furthermore, for all $s'\in[0,s),$ we have $\mu\in C([0,T];H^{s'-1})$ and $u\in C([0,T];H^{s'}).$
\end{theorem}

We commence with the proof of the existence theorem.
We provide some notation for certain norms which we will be useful for our estimates.
For all $n\in\mathbb{N},$ we define $M_{n}$ and $N_{n}$ to be
\begin{equation}\label{mndef}
M_{n}=\sup_{t\in[0,T]}\left(\|Dw^{n}\|_{s-1}^{2}+\|\mu^{n}\|_{s-1}^{2}\right),
\end{equation}
\begin{equation}\label{nndef}
N_{n}=\sum_{1\leq |\alpha|\leq s}\int_{0}^{T}\|\partial^{\alpha}Dw^{n}\|_{0}^{2}\ d\tau
+\sum_{0\leq |\alpha|\leq s-1}\int_{0}^{T}\|\partial^{\alpha}D\mu^{n}\|_{0}^{2}\ d\tau.
\end{equation}

We will be proving an estimate for the solutions which is uniform in $n.$  We will do so 
in stages; first, we will prove an estimate for $(w^{n+1},\mu^{n+1})$ in terms of $(w^{n},\mu^{n}).$
Then we will proceed inductively, making an assumption about $(w^{n},\mu^{n}),$ and showing that this
implies the corresponding bound holds for $(w^{n+1},\mu^{n+1}).$  This inductive step will use our smallness
assumption (which remains to be stated).

Let $\alpha$ be a multi-index (as defined in Section \ref{preliminaries}) of order $|\alpha|=s-1.$  We compute
the time derivative of the square of the $L^{2}$-norm of $\partial^{\alpha}\mu:$
\begin{multline}\label{muN+1First}
\frac{d}{dt}\frac{1}{2}\int_{\mathbb{T}^{d}}\left(\partial^{\alpha}\mu^{n+1}\right)^{2}\ dx
\\
=\int_{\mathbb{T}^{d}}\left(\partial^{\alpha}\mu^{n+1}\right)\left(\partial^{\alpha}\Delta\mu^{n+1}\right)\ dx
-\varepsilon\int_{\mathbb{T}^{d}}\left(\partial^{\alpha}\mu^{n+1}\right)\partial^{\alpha}
\left(D\mu^{n}\cdot\Theta_{p}(\cdot,x,\mu^{n},Dw^{n})\right)\ dx
\\
-\varepsilon\int_{\mathbb{T}^{d}}\left(\partial^{\alpha}\mu^{n+1}\right)\partial^{\alpha}
\left((\mu^{n}+\bar{m})\sum_{i=1}^{d}\Theta_{x_{i}p_{i}}(\cdot,x,\mu^{n},Dw^{n})\right)\ dx
\\
-\varepsilon\int_{\mathbb{T}^{d}}\left(\partial^{\alpha}\mu^{n+1}\right)\partial^{\alpha}
\left((\mu^{n}+\bar{m})\sum_{i=1}^{d}\left[\left(\Theta_{qp_{i}}(\cdot,x,\mu^{n},Dw^{n})\right)
\left(\partial_{x_{i}}\mu^{n}\right)\right]\right)\ dx\\
-\varepsilon\int_{\mathbb{T}^{d}}
\left(\partial^{\alpha}\mu^{n+1}\right)\partial^{\alpha}\left((\mu^{n}+\bar{m})\sum_{i=1}^{d}\sum_{j=1}^{d}
\left[\left(\Theta_{p_{i}p_{j}}(\cdot,x,\mu^{n},Dw^{n})\right)
\left(\partial_{x_{i}x_{j}}^{2}w^{n}\right)\right]\right)\ dx.
\end{multline}
We integrate by parts in the first integral on the right-hand side, we move the resulting integral to the left-hand side, 
and we integrate \eqref{muN+1First} in time, over the interval $[0,t]:$
\begin{multline}\label{muN+1Second}
\frac{1}{2}\int_{\mathbb{T}^{d}}\left(\partial^{\alpha}\mu^{n+1}(t,x)\right)^{2}\ dx
-\frac{1}{2}\int_{\mathbb{T}^{d}}\left(\partial^{\alpha}\mu^{n+1}(0,x)\right)^{2}\ dx
+\int_{0}^{t}\int_{\mathbb{T}^{d}}\left|D\partial^{\alpha}\mu^{n+1}\right|^{2}\ dx d\tau
\\
=
-\varepsilon\int_{0}^{t}\int_{\mathbb{T}^{d}}\left(\partial^{\alpha}\mu^{n+1}\right)\partial^{\alpha}
\left(D\mu^{n}\cdot\Theta_{p}(\tau,x,\mu^{n},Dw^{n})\right)\ dx d\tau
\\
-\varepsilon\int_{0}^{t}\int_{\mathbb{T}^{d}}\left(\partial^{\alpha}\mu^{n+1}\right)\partial^{\alpha}
\left((\mu^{n}+\bar{m})\sum_{i=1}^{d}\Theta_{x_{i}p_{i}}(\tau,x,\mu^{n},Dw^{n})\right)\ dx d\tau
\\
-\varepsilon\int_{0}^{t}\int_{\mathbb{T}^{d}}\left(\partial^{\alpha}\mu^{n+1}\right)\partial^{\alpha}
\left((\mu^{n}+\bar{m})\sum_{i=1}^{d}\left[\left(\Theta_{qp_{i}}(\tau,x,\mu^{n},Dw^{n})\right)
\left(\partial_{x_{i}}\mu^{n}\right)\right]\right)\ dx d\tau\\
-\varepsilon\int_{0}^{t}\int_{\mathbb{T}^{d}}
\left(\partial^{\alpha}\mu^{n+1}\right)\partial^{\alpha}\left((\mu^{n}+\bar{m})\sum_{i=1}^{d}\sum_{j=1}^{d}
\left[\left(\Theta_{p_{i}p_{j}}(\tau,x,\mu^{n},Dw^{n})\right)
\left(\partial_{x_{i}x_{j}}^{2}w^{n}\right)\right]\right)\ dx d\tau
\\
=I+II+III+IV.
\end{multline}

We first work to estimate $I,$  and we begin by adding and subtracting:
\begin{multline}\nonumber
I=-\varepsilon\int_{0}^{t}\int_{\mathbb{T}^{d}}(\partial^{\alpha}\mu^{n+1})
(\partial^{\alpha}D\mu^{n})\cdot\Theta_{p}(\tau,x,\mu^{n},Dw^{n})\ dxd\tau
\\
+\varepsilon\int_{0}^{t}\int_{\mathbb{T}^{d}}(\partial^{\alpha}\mu^{n+1})\left[
(\partial^{\alpha}D\mu^{n})\cdot\Theta_{p}(\tau,x,\mu^{n},Dw^{n})
-\partial^{\alpha}(D\mu^{n}\cdot\Theta_{p}(\tau,x,\mu^{n},Dw^{n}))\right]
\ dxd\tau
\\
=I_{A}+I_{B}.
\end{multline}
We start with $I_{A},$ pulling the supremum of the $\Theta_{p}$ term outside the integrals:
\begin{equation}\nonumber
I_{A}\leq \varepsilon\left(\sup_{t\in[0,T]}|\Theta_{p}(t,\cdot,\mu^{n},Dw^{n})|_{\infty}\right)
\int_{0}^{t}\int_{\mathbb{T}^{d}}|\partial^{\alpha}\mu^{n+1}||\partial^{\alpha}D\mu^{n}|\ dxd\tau.
\end{equation}
We use \eqref{thetaInfinityBound} to bound $\Theta_{p}$ in terms of $M_{n}:$
\begin{equation}\nonumber
I_{A}\leq \varepsilon F(M_{n})
\int_{0}^{t}\int_{\mathbb{T}^{d}}|\partial^{\alpha}\mu^{n+1}||\partial^{\alpha}D\mu^{n}|\ dxd\tau.
\end{equation}

Next, we continue by using \eqref{young} with positive parameter $\sigma_{1},$ which will be determined presently:
\begin{multline}\nonumber
I_{A}\leq \varepsilon F(M_{n})
\left(\frac{1}{2\sigma_{1}}\int_{0}^{T}\|\partial^{\alpha}\mu^{n+1}\|_{0}^{2}\ d\tau
+\frac{\sigma_{1}}{2}\int_{0}^{T}\|\partial^{\alpha}D\mu^{n}\|_{0}^{2}\ d\tau\right)
\\ \leq
\varepsilon F(M_{n})
\left(\frac{1}{2\sigma_{1}}\int_{0}^{T}\|\partial^{\alpha}\mu^{n+1}\|_{0}^{2}\ d\tau
+\frac{\sigma_{1}}{2}N_{n}\right).
\end{multline}
We let $\sigma_{1}=28T\varepsilon F(M_{n}),$ and this choice then yields the following:
\begin{multline}\label{IAConclusion}
I_{A}\leq \frac{1}{56 T}\int_{0}^{T}\|\partial^{\alpha}\mu^{n+1}\|_{0}^{2}\ d\tau
+14\varepsilon^{2}T(F(M_{n}))^{2}N_{n}\\
\leq
\frac{1}{56}\left(\sup_{t\in[0,T]}\|\partial^{\alpha}\mu^{n+1}\|_{0}^{2}\right)
+14\varepsilon^{2}T(F(M_{n}))^{2}N_{n}.
\end{multline}

We turn now to estimating $I_{B};$ we start by using \eqref{young} with parameter $\sigma_{2}:$
\begin{multline}\nonumber
I_{B}\leq \varepsilon\int_{0}^{t}\int_{\mathbb{T}^{d}}|(\partial^{\alpha}\mu^{n+1})|
\Big|
(\partial^{\alpha}D\mu^{n})\cdot\Theta_{p}(\tau,x,\mu^{n},Dw^{n})
-\partial^{\alpha}(D\mu^{n}\cdot\Theta_{p}(\tau,x,\mu^{n},Dw^{n}))\Big|
\ dxd\tau
\\
\leq 
\varepsilon\int_{0}^{T}\frac{1}{2\sigma_{2}}\|\partial^{\alpha}\mu^{n+1}\|_{0}^{2} d\tau
\\
+\frac{\varepsilon\sigma_{2}}{2}\int_{0}^{T}\Big\|
(\partial^{\alpha}D\mu^{n})\cdot\Theta_{p}(\tau,x,\mu^{n},Dw^{n})
-\partial^{\alpha}(D\mu^{n}\cdot\Theta_{p}(\tau,x,\mu^{n},Dw^{n}))\Big\|_{0}^{2}\ d\tau.
\end{multline}
We let $\sigma_{2}=28T\varepsilon,$ and we continue:
\begin{multline}\nonumber
I_{B}\leq \frac{1}{56}\left(\sup_{t\in[0,T]}\|\partial^{\alpha}\mu^{n+1}\|_{0}^{2}\right)
\\
+14\varepsilon^{2}T\int_{0}^{T}\Big\|
(\partial^{\alpha}D\mu^{n})\cdot\Theta_{p}(\tau,\cdot,\mu^{n},Dw^{n})
-\partial^{\alpha}(D\mu^{n}\cdot\Theta_{p}(\tau,\cdot,\mu^{n},Dw^{n}))\Big\|_{0}^{2}\ d\tau.
\end{multline}
Next, we use Lemma \ref{leadingTermsSobolev} and Sobolev embedding as follows (we temporarily suppress
the dependence of $\Theta_{p}$ on its arguments):
\begin{multline}\nonumber
\|\partial^{\alpha}(D\mu^{n}\cdot\Theta_{p})-(\partial^{\alpha}D\mu^{n})\cdot\Theta_{p}\|_{0}^{2}
\leq c\left(|D\Theta_{p}|_{\infty}\|D^{s-1}\mu^{n}\|_{0}+\|D^{s-1}\Theta_{p}\|_{0}|D\mu^{n}|_{\infty}\right)^{2}\\
\leq c\left(\|\Theta_{p}\|_{\left\lceil\frac{d+3}{2}\right\rceil}\|\mu^{n}\|_{s-1}
+\|\Theta_{p}\|_{s-1}\|\mu^{n}\|_{\left\lceil\frac{d+3}{2}\right\rceil}\right)^{2}
\leq c\|\Theta_{p}\|_{s-1}^{2}\|\mu^{n}\|_{s-1}^{2}.
\end{multline}
Here, we have used the condition $s\geq\left\lceil\frac{d+5}{2}\right\rceil.$
By Lemma \ref{boundsOnTheta}, we have
\begin{equation}\nonumber
\|\Theta_{p}\|_{s-1}\leq cF(M_{n})(1+M_{n})^{(s-1)/2}.
\end{equation}
Putting this information together, we complete our bound of $I_{B}:$
\begin{equation}\label{IBConclusion}
I_{B}\leq \frac{1}{56}\left(\sup_{t\in[0,T]}\|\partial^{\alpha}\mu^{n+1}\|_{0}^{2}\right)
+c\varepsilon^{2}T^{2}(F(M_{n}))^{2}M_{n}(1+M_{n})^{s-1}.
\end{equation}
Note that this constant $c$ is independent of $\varepsilon,$ $T,$ $n,$ and $\delta;$ 
instead it depends only upon $s$ and $d.$ 
The same will be true for any constants which we call $c$ in the sequel.

We are ready to estimate the term $II.$  We begin by using Young's inequality \eqref{young} with parameter
$\sigma_{3}=28T\varepsilon:$
\begin{multline}\label{firstStepForBoundingII}
II\leq\varepsilon\int_{0}^{T}\int_{\mathbb{T}^{d}}\frac{|\partial^{\alpha}\mu^{n+1}|^{2}}{2\sigma_{3}}
+\frac{\sigma_{3}}{2}\left|\partial^{\alpha}\left((\mu^{n}+\bar{m})\sum_{i=1}^{d}
\Theta_{x_{i}p_{i}}(\tau,x,\mu^{n},Dw^{n})\right)\right|^{2}\ dxd\tau
\\
\leq\frac{1}{56}\left(\sup_{t\in[0,T]}\|\partial^{\alpha}\mu^{n+1}\|_{0}^{2}\right)
+14\varepsilon^{2}T\int_{0}^{T}
\left\|\partial^{\alpha}\left((\mu^{n}+\bar{m})\sum_{i=1}^{d}\Theta_{x_{i}p_{i}}(\tau,\cdot,\mu^{n},Dw^{n})\right)\right\|_{0}^{2}
\ d\tau.
\end{multline}
We use the Sobolev algebra property and Lemma \ref{boundsOnTheta} to bound the integrand:
\begin{multline}\label{intermediateStepForBoundingII}
\left\|\partial^{\alpha}\left((\mu^{n}+\bar{m})\sum_{i=1}^{d}\Theta_{x_{i}p_{i}}(\tau,\cdot,\mu^{n},Dw^{n})\right)\right\|_{0}^{2}
\leq \left\|(\mu^{n}+\bar{m})\sum_{i=1}^{d}\Theta_{x_{i}p_{i}}(\tau,\cdot,\mu^{n},Dw^{n})\right\|_{s-1}^{2}
\\
\leq c\|\mu^{n}+\bar{m}\|_{s-1}^{2}(F(M_{n}))^{2}(1+M_{n})^{s-1}
\leq c(F(M_{n}))^{2}(1+M_{n})^{s}.
\end{multline}
Combining \eqref{intermediateStepForBoundingII} with \eqref{firstStepForBoundingII}, we complete our bound for the term 
$II:$
\begin{equation}\label{IIConclusion}
II\leq \frac{1}{56}\left(\sup_{t\in[0,T]}\|\partial^{\alpha}\mu^{n+1}\|_{0}^{2}\right)+
c\varepsilon^{2}T^{2}(F(M_{n}))^{2}(1+M_{n})^{s}.
\end{equation}

Before estimating $III,$ we again add and subtract to isolate the leading-order term.  We have
$III=III_{A}+III_{B},$ with $III_{A}$ given by
\begin{equation}\nonumber
III_{A}=
-\varepsilon\int_{0}^{t}\int_{\mathbb{T}^{d}}\left(\partial^{\alpha}\mu^{n+1}\right)
(\mu^{n}+\bar{m})\sum_{i=1}^{d}\left[\left(\Theta_{qp_{i}}(\tau,x,\mu^{n},Dw^{n})\right)
\left(\partial^{\alpha}\partial_{x_{i}}\mu^{n}\right)\right]\ dx d\tau,
\end{equation}
and the remainder $III_{B}$ given by
\begin{multline}\nonumber
III_{B}=
-\varepsilon\int_{0}^{t}\int_{\mathbb{T}^{d}}\left(\partial^{\alpha}\mu^{n+1}\right)\Bigg\{\partial^{\alpha}
\left((\mu^{n}+\bar{m})\sum_{i=1}^{d}\left[\left(\Theta_{qp_{i}}(\tau,x,\mu^{n},Dw^{n})\right)
\left(\partial_{x_{i}}\mu^{n}\right)\right]\right)
\\
-(\mu^{n}+\bar{m})\sum_{i=1}^{d}\left[\left(\Theta_{qp_{i}}(\tau,x,\mu^{n},Dw^{n})\right)
\left(\partial^{\alpha}\partial_{x_{i}}\mu^{n}\right)\right]\Bigg\}\ dx d\tau.
\end{multline}
We begin estimating $III_{A}$ by using the triangle inequality and pulling the lower-order terms through the integrals:
\begin{multline}\nonumber
III_{A}\leq c\varepsilon(1+M_{n})^{1/2}\left(\sup_{t\in[0,T],i\in\{1,\ldots,d\}}|\Theta_{qp_{i}}(t,\cdot,\mu^{n},Dw^{n})|_{\infty}\right)
\\
\times
\sum_{i=1}^{d}\int_{0}^{T}\int_{\mathbb{T}^{d}}|\partial^{\alpha}\mu^{n+1}||\partial^{\alpha}\partial_{x_{i}}\mu^{n}|
\ dxd\tau.
\end{multline}
Next, we use \eqref{thetaInfinityBound}, and we use Young's inequality \eqref{young} with parameter $\sigma_{4}:$
\begin{equation}\nonumber
III_{A}\leq c\varepsilon(1+M_{n})^{1/2}F(M_{n})\int_{0}^{T}\frac{1}{2\sigma_{4}}\|\partial^{\alpha}\mu^{n+1}\|_{0}^{2}
+\frac{\sigma_{4}}{2}\|\partial^{\alpha}D\mu^{n}\|_{0}^{2}\ d\tau.
\end{equation}
We take $\sigma_{4}=28c\varepsilon T(1+M_{n})^{1/2}F(M_{n}),$ and we complete our estimate of $III_{A}:$
\begin{equation}\label{IIIAConclusion}
III_{A}\leq\frac{1}{56}\left(\sup_{t\in[0,T]}\|\partial^{\alpha}\mu^{n+1}\|_{0}^{2}\right)
+c\varepsilon^{2}T(1+M_{n})(F(M_{n}))^{2}N_{n}.
\end{equation}

We next estimate $III_{B}.$  We begin by applying Young's inequality, with parameter $\sigma_{5}>0:$
\begin{multline}\nonumber
III_{B}\leq \varepsilon\int_{0}^{t}\int_{\mathbb{T}^{d}}\frac{1}{2\sigma_{5}}|\partial^{\alpha}\mu^{n+1}|^{2}\\
+\frac{\sigma_{5}}{2}\left|
\partial^{\alpha}\left((\mu^{n}+m)\sum_{i=1}^{d}(\Theta_{qp_{i}})(\partial_{x_{i}}\mu^{n})\right)
-(\mu^{n}+\bar{m})\sum_{i=1}^{d}(\Theta_{qp_{i}})(\partial^{\alpha}\partial_{x_{i}}\mu^{n})
\right|^{2}\ dxd\tau
\\
=\frac{\varepsilon}{2\sigma_{5}}\int_{0}^{t}\|\partial^{\alpha}\mu^{n+1}\|_{0}^{2}\ d\tau
\\
+\frac{\varepsilon\sigma_{5}}{2}\int_{0}^{t}
\left\|\partial^{\alpha}\left((\mu^{n}+\bar{m})\sum_{i=1}^{d}(\Theta_{qp_{i}})(\partial_{x_{i}}\mu^{n})\right)
-(\mu^{n}+\bar{m})\sum_{i=1}^{d}(\Theta_{qp_{i}})(\partial^{\alpha}\partial_{x_{i}}\mu^{n})\right\|_{0}^{2}\ d\tau.
\end{multline}
We proceed by using Lemma \ref{leadingTermsSobolev}, as in our previous estimate for the term $I_{B};$
we find the following:  
\begin{equation}\nonumber
III_{B}\leq\frac{\varepsilon T}{2\sigma_{5}}\left(\sup_{t\in[0,T]}\|\partial^{\alpha}\mu^{n+1}\|_{0}^{2}\right)
+cT\varepsilon\sigma_{5}(F(M_{n}))^{2}(1+M_{n})^{s+1}.
\end{equation}
We conclude the estimate of $III_{B}$ by setting $\sigma_{5}=28T\varepsilon:$
\begin{equation}\label{IIIBConclusion}
III_{B}\leq\frac{1}{56}\left(\sup_{t\in[0,T]}\|\partial^{\alpha}\mu^{n+1}\|_{0}^{2}\right)
+cT^{2}\varepsilon^{2}(F(M_{n}))^{2}(1+M_{n})^{s+1}.
\end{equation}

For the term $IV,$ we again must separate the leading-order term by adding and subtracting.
We write $IV=IV_{A}+IV_{B},$ with $IV_{A}$ given by
\begin{equation}\nonumber
IV_{A}=-\varepsilon\int_{0}^{t}\int_{\mathbb{T}^{d}}
\left(\partial^{\alpha}\mu^{n+1}\right)
(\mu^{n}+\bar{m})\sum_{i=1}^{d}\sum_{j=1}^{d}
\left[\left(\Theta_{p_{i}p_{j}}(\tau,x,\mu^{n},Dw^{n})\right)
\left(\partial^{\alpha}\partial_{x_{i}x_{j}}^{2}w^{n}\right)\right]\ dx d\tau,
\end{equation}
and with the remainder $IV_{B}$ given by
\begin{multline}\nonumber
IV_{B}=
\\
-\varepsilon\int_{0}^{t}\int_{\mathbb{T}^{d}}
\left(\partial^{\alpha}\mu^{n+1}\right)\Bigg\{
\partial^{\alpha}\left((\mu^{n}+\bar{m})\sum_{i=1}^{d}\sum_{j=1}^{d}
\left[\left(\Theta_{p_{i}p_{j}}(\tau,x,\mu^{n},Dw^{n})\right)
\left(\partial_{x_{i}x_{j}}^{2}w^{n}\right)\right]\right)\
\\
-(\mu^{n}+\bar{m})\sum_{i=1}^{d}\sum_{j=1}^{d}
\left[\left(\Theta_{p_{i}p_{j}}(\tau,x,\mu^{n},Dw^{n})\right)
\left(\partial^{\alpha}\partial_{x_{i}x_{j}}^{2}w^{n}\right)\right]\Bigg\}
\ dxd\tau.
\end{multline}

To estimate $IV_{A},$ we first pull $\mu^{n}+\bar{m}$ and $\Theta_{p_{i}p_{j}}$ through the integrals by taking
supremums:
\begin{equation}\nonumber
IV_{A}\leq c\varepsilon(1+M_{n})^{1/2}\left(\sup_{t\in[0,T]}\sup_{i,j}|\Theta_{p_{i}p_{j}}|_{\infty}\right)
\sum_{i,j}\int_{0}^{t}\int_{\mathbb{T}^{d}}|\partial^{\alpha}\mu^{n+1}|
|\partial^{\alpha}\partial^{2}_{x_{i}x_{j}}w^{n}|\ dxd\tau.
\end{equation}
We estimate this by using \eqref{thetaInfinityBound}, and we use Young's inequality with parameter $\sigma_{6}>0:$
\begin{equation}\nonumber
IV_{A}\leq c\varepsilon F(M_{n})(1+M_{n})^{1/2}\sum_{i,j}\int_{0}^{T}\frac{1}{2\sigma_{6}}\|\partial^{\alpha}\mu^{n+1}\|_{0}^{2}
+\frac{\sigma_{6}}{2}\|\partial^{\alpha}\partial^{2}_{x_{i}x_{j}}w^{n}\|_{0}^{2}\ d\tau.
\end{equation}
Taking $\sigma_{6}=28c\varepsilon T(1+M_{n})^{1/2}F(M_{n}),$ and proceeding as we have previously, we arrive at
our final bound for $IV_{A}:$
\begin{equation}\label{IVAConclusion}
IV_{A}\leq\frac{1}{56}\left(\sup_{t\in[0,T]}\|\partial^{\alpha}\mu^{n+1}\|_{0}^{2}\right)
+c\varepsilon^{2}T(1+M_{n})(F(M_{n}))^{2}N_{n}.
\end{equation}

We estimate $IV_{B}$ just as we have estimated $III_{B},$ finding that
\begin{equation}\label{IVBConclusion}
IV_{B}\leq\frac{1}{56}\left(\sup_{t\in[0,T]}\|\partial^{\alpha}\mu^{n+1}\|_{0}^{2}\right)
+cT^{2}\varepsilon^{2}(F(M_{n}))^{2}(1+M_{n})^{s+1}.
\end{equation}
To summarize our progress so far, we add \eqref{IAConclusion}, \eqref{IBConclusion}, \eqref{IIConclusion},
\eqref{IIIAConclusion}, \eqref{IIIBConclusion}, \eqref{IVAConclusion}, and \eqref{IVBConclusion}, and we make 
some elementary bounds, to find the following:
\begin{multline}\label{firstSumConclusion}
I+II+III+IV\\
\leq \frac{1}{8}\left(\sup_{t\in[0,T]}\|\partial^{\alpha}\mu^{n+1}\|_{0}^{2}\right)
+c\varepsilon^{2}T(F(M_{n}))^{2}\Bigg((1+T)(1+N_{n})(1+M_{n})^{s+1}\Bigg).
\end{multline}

Continuing, we compute the time derivative of the square of the $L^{2}$-norm of 
$\partial^{\alpha}\partial_{x_{j}}w^{n+1},$ substituting from \eqref{uN+1xJEvolution}:
\begin{multline}\label{ddtuN+1first}
\frac{d}{dt}\frac{1}{2}\int_{\mathbb{T}^{d}}(\partial^{\alpha}\partial_{x_{j}}w^{n+1})^{2}\ dx
=-\int_{\mathbb{T}^{d}}(\partial^{\alpha}\partial_{x_{j}}w^{n+1})(\partial^{\alpha}\Delta\partial_{x_{j}}w^{n+1})\ dx
\\
-\varepsilon\int_{\mathbb{T}^{d}}(\partial^{\alpha}\partial_{x_{j}}w^{n+1})
\partial^{\alpha}\left(\Theta_{x_{j}}(\cdot,x,\mu^{n},Dw^{n})\right)\ dx
\\
-\varepsilon\int_{\mathbb{T}^{d}}(\partial^{\alpha}\partial_{x_{j}}w^{n+1})
\partial^{\alpha}\left(\left(\Theta_{q}(\cdot,x,\mu^{n},Dw^{n})\right)\mu^{n}_{x_{j}}\right)\ dx
\\
-\varepsilon\int_{\mathbb{T}^{d}}(\partial^{\alpha}\partial_{x_{j}}w^{n+1})
\partial^{\alpha}\left(\sum_{i=1}^{d}\left[
(\Theta_{p_{i}}(\cdot,x,\mu^{n},Dw^{n}))\partial^{2}_{x_{i}x_{j}}w^{n}\right]\right)\ dx.
\end{multline}
We integrate by parts in the first integral on the right-hand side and
we integrate \eqref{ddtuN+1first} in time over the interval $[t,T];$ we also rearrange terms slightly, arriving at the
following:
\begin{multline}\label{ddtUN+1first}
\frac{1}{2}\int_{\mathbb{T}^{d}}(\partial^{\alpha}\partial_{x_{j}}w^{n+1}(t,x))^{2}\ dx
-\frac{1}{2}\int_{\mathbb{T}^{d}}(\partial^{\alpha}\partial_{x_{j}}w^{n+1}(T,x))^{2}\ dx\\
+\int_{t}^{T}\int_{\mathbb{T}^{d}}|D\partial^{\alpha}\partial_{x_{j}}w^{n+1}|^{2}
\ dxd\tau
\\
=
\varepsilon\int_{t}^{T}\int_{\mathbb{T}^{d}}(\partial^{\alpha}\partial_{x_{j}}w^{n+1})
\partial^{\alpha}\left(\Theta_{x_{j}}(\cdot,x,\mu^{n},Dw^{n})\right)\ dxd\tau
\\
+\varepsilon\int_{t}^{T}\int_{\mathbb{T}^{d}}(\partial^{\alpha}\partial_{x_{j}}w^{n+1})
\partial^{\alpha}\left(\left(\Theta_{q}(\cdot,x,\mu^{n},Dw^{n})\right)\mu^{n}_{x_{j}}\right)\ dxd\tau
\\
+\varepsilon\int_{t}^{T}\int_{\mathbb{T}^{d}}(\partial^{\alpha}\partial_{x_{j}}w^{n+1})
\partial^{\alpha}\left(\sum_{i=1}^{d}\left[
(\Theta_{p_{i}}(\cdot,x,\mu^{n},Dw^{n}))\partial^{2}_{x_{i}x_{j}}w^{n}\right]\right)\ dxd\tau
\\
=V+VI+VII.
\end{multline}

The term $V$ is straightforward to estimate; we begin with Young's inequality, with parameter $\sigma_{7}>0:$
\begin{equation}\nonumber
V\leq\varepsilon\int_{t}^{T}\int_{\mathbb{T}^{d}}\frac{1}{2\sigma_{7}}|\partial^{\alpha}\partial_{x_{j}}w^{n+1}|^{2}
+\frac{\sigma_{7}}{2}|\partial^{\alpha}\Theta_{x_{j}}|^{2}\ dxd\tau.
\end{equation}
We choose $\sigma_{7}=20\varepsilon T,$ we use Lemma \ref{boundsOnTheta}, 
and we estimate similarly to the previous terms to find the following: 
\begin{equation}\label{VConclusion}
V\leq \frac{1}{40}\left(\sup_{t\in[0,T]}\|\partial^{\alpha}\partial_{x_{j}}w^{n+1}\|_{0}^{2}\right)
+c\varepsilon^{2} T^{2}(F(M_{n}))^{2}(1+M_{n})^{s-1}.
\end{equation}

Before estimating $VI,$ we must add and subtract to isolate the leading-order term.  We have $VI=VI_{A}+VI_{B},$
with $VI_{A}$ given by
\begin{equation}\nonumber
VI_{A}=\varepsilon\int_{t}^{T}\int_{\mathbb{T}^{d}}
(\partial^{\alpha}\partial_{x_{j}}w^{n+1})
(\Theta_{q}(\tau,x,\mu^{n},Dw^{n}))(\partial^{\alpha}\partial_{x_{j}}\mu^{n})\ dxd\tau,
\end{equation}
and with the remainder $VI_{B}$ given by 
\begin{multline}\nonumber
VI_{B}=\varepsilon\int_{t}^{T}\int_{\mathbb{T}^{d}}
(\partial^{\alpha}\partial_{x_{j}}w^{n+1})\Bigg\{
\partial^{\alpha}\left(\left(\Theta_{q}(\tau,x,\mu^{n},Dw^{n})\right)
\left(\partial_{x_{j}}\mu^{n}\right)\right)
\\
-
\left(\Theta_{q}(\tau,x,\mu^{n},Dw^{n})\right)
\left(\partial^{\alpha}\partial_{x_{j}}\mu^{n}\right)\Bigg\}
\ dxd\tau.
\end{multline}

We begin estimating $VI_{A}$ by taking the supremum of $\Theta_{q}$ with respect to space and time, and pulling 
this through the integrals:
\begin{equation}\nonumber
VI_{A}\leq\varepsilon\left(\sup_{t\in[0,T]}|\Theta_{q}(t,\cdot,\mu^{n},Dw^{n})|_{\infty}\right)
\int_{t}^{T}\int_{\mathbb{T}^{d}}|\partial^{\alpha}\partial_{x_{j}}w^{n+1}||\partial^{\alpha}\partial_{x_{j}}\mu^{n}|\ dxd\tau.
\end{equation}
We then bound the $\Theta_{q}$ term by using \eqref{thetaInfinityBound} and by using Young's inequality with positive
parameter $\sigma_{8}:$
\begin{equation}\nonumber
VI_{A}\leq \varepsilon F(M_{n})\int_{t}^{T}\frac{1}{2\sigma_{8}}\|\partial^{\alpha}\partial_{x_{j}}w^{n+1}\|_{0}^{2}
+\frac{\sigma_{8}}{2}\|\partial^{\alpha}\partial_{x_{j}}\mu^{n}\|_{0}^{2}\ d\tau.
\end{equation}
We take $\sigma_{8}=20\varepsilon T F(M_{n}),$ and estimate as we have previously, finding the following:
\begin{equation}\label{VIAConclusion}
VI_{A}\leq\frac{1}{40}\left(\sup_{t\in[0,T]}\|\partial^{\alpha}\partial_{x_{j}}w^{n+1}\|_{0}^{2}\right)
+c\varepsilon^{2} T(F(M_{n}))^{2}N_{n}.
\end{equation}

We begin estimating $VI_{B}$ with use of Young's inequality, with positive parameter $\sigma_{9}:$
\begin{equation}\nonumber
VI_{B}\leq\varepsilon\int_{t}^{T}\int_{\mathbb{T}^{d}}\frac{1}{2\sigma_{9}}|\partial^{\alpha}\partial_{x_{j}}w^{n+1}|^{2}
+\frac{\sigma_{9}}{2}\left|\partial^{\alpha}((\Theta_{q})(\partial_{x_{j}}\mu^{n}))
-(\Theta_{q})(\partial^{\alpha}\partial_{x_{j}}\mu^{n})\right|^{2}
\ dxd\tau.
\end{equation}
We take $\sigma_{9}=20\varepsilon T,$ and proceed as usual:
\begin{multline}\nonumber
VI_{B}\leq
\frac{1}{40}\left(\sup_{t\in[0,T]}\|\partial^{\alpha}\partial_{x_{j}}w^{n+1}\|_{0}^{2}\right)
\\
+c\varepsilon^{2}T^{2}\left(\sup_{t\in[0,T]}
\left\|\partial^{\alpha}((\Theta_{q})(\partial_{x_{j}}\mu^{n}))-(\Theta_{q})(\partial^{\alpha}\partial_{x_{j}}\mu^{n})\right\|_{0}^{2}
\right).
\end{multline}
Using Lemma \ref{leadingTermsSobolev} and Sobolev embedding, we bound this as follows:
\begin{multline}\nonumber
VI_{B}\leq\frac{1}{40}\left(\sup_{t\in[0,T]}\|\partial^{\alpha}\partial_{x_{j}}w^{n+1}\|_{0}^{2}\right)
\\
+c\varepsilon^{2}T^{2}\left(\sup_{t\in[0,T]}
\|\Theta_{q}\|_{\left\lceil\frac{d+3}{2}\right\rceil}^{2}\|\mu^{n}\|_{s-1}^{2}
+\|\Theta_{q}\|_{s-1}^{2}\|\mu^{n}\|_{\left\lceil\frac{d+3}{2}\right\rceil}^{2}
\right).
\end{multline}
Using Lemma \ref{boundsOnTheta}, and the fact that $s$ is sufficiently large ($s\geq\left\lceil\frac{d+5}{2}\right\rceil$ is
needed here), we conclude our bound of $VI_{B}:$
\begin{equation}\label{VIBConclusion}
VI_{B}\leq
\frac{1}{40}\left(\sup_{t\in[0,T]}\|\partial^{\alpha}\partial_{x_{j}}w^{n+1}\|_{0}^{2}\right)
+c\varepsilon^{2}T^{2}(F(M_{n}))^{2}(1+M_{n})^{s}.
\end{equation}

For the term $VII,$ we must again add and subtract to isolate the leading-order contribution.  We write
$VII=VII_{A}+VII_{B},$ with $VII_{A}$ given by
\begin{equation}\nonumber
VII_{A}=\varepsilon\int_{t}^{T}\int_{\mathbb{T}^{d}}
(\partial^{\alpha}\partial_{x_{j}}w^{n+1})
\sum_{i=1}^{d}\left[\left(\Theta_{p_{i}}(\tau,x,\mu^{n},Dw^{n})\right)
\left(\partial^{\alpha}\partial_{x_{i}x_{j}}^{2}w^{n}\right)\right]\ dxd\tau,
\end{equation}
and with the remainder $VII_{B}$ given by
\begin{multline}\nonumber
VII_{B}=\varepsilon\int_{t}^{T}\int_{\mathbb{T}^{d}}
(\partial^{\alpha}\partial_{x_{j}}w^{n+1})\Bigg\{
\partial_{\alpha}\sum_{i=1}^{d}\left[\left(\Theta_{p_{i}}(\tau,x,\mu^{n},Dw^{n})\right)\partial^{2}_{x_{i}x_{j}}w^{n}\right]
\\
-\sum_{i=1}^{d}\left[\left(\Theta_{p_{i}}(\tau,x,\mu^{n},Dw^{n})\right)\left(\partial^{\alpha}\partial^{2}_{x_{i}x_{j}}w^{n}\right)
\right]\Bigg\}\ dxd\tau.
\end{multline}

To begin to estimate $VII_{A},$ we pull $\Theta_{p_{i}}$ through the integrals (after taking its supremum):
\begin{multline}\nonumber
VII_{A}\leq \sum_{i=1}^{d} \varepsilon\left(\sup_{t\in[0,T]}|\Theta_{p_{i}}(t,\cdot,\mu^{n},Dw^{n})|_{\infty}\right)\\
\times
\int_{t}^{T}\int_{\mathbb{T}^{d}}
|\partial^{\alpha}\partial_{x_{j}}w^{n+1}||\partial^{\alpha}\partial^{2}_{x_{i}x_{j}}w^{n}|\ dxd\tau.
\end{multline}
We estimate $\Theta_{p_{i}}$ by using \eqref{thetaInfinityBound}, 
and we use Young's inequality with positive parameter $\sigma_{10}:$
\begin{equation}\nonumber
VII_{A}\leq \sum_{i=1}^{d} \varepsilon F(M_{n})
\int_{t}^{T}\int_{\mathbb{T}^{d}}
\frac{1}{2\sigma_{10}}|\partial^{\alpha}\partial_{x_{j}}w^{n+1}|^{2}
+\frac{\sigma_{10}}{2}|\partial^{\alpha}\partial^{2}_{x_{i}x_{j}}w^{n}|\ dxd\tau.
\end{equation}
We choose the value $\sigma_{10}=20d\varepsilon TF(M_{n}),$ and thus find the following bound:
\begin{equation}\label{VIIAConclusion}
VII_{A}\leq \frac{1}{40}\left(\sup_{t\in[0,T]}\|\partial^{\alpha}\partial_{x_{j}}w^{n+1}\|_{0}^{2}\right)
+c\varepsilon^{2}(F(M_{n}))^{2}TN_{n}.
\end{equation}

We now estimate the final term, $VII_{B}.$  We interchange the summation and the integrals, and we use Young's inequality
with positive parameter $\sigma_{11}:$
\begin{multline}\nonumber
VII_{B}\leq \varepsilon\sum_{i=1}^{d}\int_{t}^{T}\int_{\mathbb{T}^{d}}\frac{1}{2\sigma_{11}}
|\partial^{\alpha}\partial_{x_{j}}w^{n+1}|^{2}
\\
+\frac{\sigma_{11}}{2}|\partial^{\alpha}((\Theta_{p_{i}})(\partial^{2}_{x_{i}x_{j}}w^{n}))
-(\Theta_{p_{i}})(\partial^{\alpha}\partial^{2}_{x_{i}x_{j}}w^{n})|^{2}\ dxd\tau
\\
\leq
\frac{\varepsilon T d}{2\sigma_{11}}\left(\sup_{t\in[0,T]}\|\partial^{\alpha}\partial_{x_{j}}w^{n+1}\|_{0}^{2}\right)
\\
+\frac{\sigma_{11}\varepsilon T}{2}\sum_{i=1}^{d}\left(\sup_{t\in[0,T]}
\|\partial^{\alpha}((\Theta_{p_{i}})(\partial^{2}_{x_{i}x_{j}}w^{n}))
-(\Theta_{p_{i}})(\partial^{\alpha}\partial^{2}_{x_{i}x_{j}}w^{n})\|_{0}^{2}\right).
\end{multline}
We thus choose $\sigma_{11}=20\varepsilon T d,$ and we use Lemma \ref{leadingTermsSobolev} and
Lemma \ref{boundsOnTheta} as we have previously:
\begin{equation}\label{VIIBConclusion}
VII_{B}\leq \frac{1}{40}\left(\sup_{t\in[0,T]}\|\partial^{\alpha}\partial_{x_{j}}w^{n+1}\|_{0}^{2}\right)
+c\varepsilon^{2}T^{2}(F(M_{n}))^{2}(1+M_{n})^{s}.
\end{equation}

We are now in a position to add \eqref{VConclusion}, \eqref{VIAConclusion}, \eqref{VIBConclusion},
\eqref{VIIAConclusion}, and \eqref{VIIBConclusion}; we then make some elementary estimates, finding the
following:
\begin{multline}\label{secondSumConclusion}
V+VI+VII\\
\leq\frac{1}{8}\left(\sup_{t\in[0,T]}\|\partial^{\alpha}\partial_{x_{j}}w^{n+1}\|_{0}^{2}\right)
+c\varepsilon^{2}T(F(M_{n}))^{2}\Bigg((1+T)(1+N_{n})(1+M_{n})^{s+1}\Bigg).
\end{multline}

We return to \eqref{muN+1Second}, considering \eqref{firstSumConclusion}.
We isolate the first term on the left-hand side of \eqref{muN+1Second}, finding the following bound:
\begin{multline}\nonumber
\frac{1}{2}\|\partial^{\alpha}\mu^{n+1}(t,\cdot)\|_{0}^{2}\leq \frac{1}{2}\|\partial^{\alpha}\mu^{n+1}(0,\cdot)\|_{0}^{2}
+\frac{1}{8}\left(\sup_{t\in[0,T]}\|\partial^{\alpha}\mu^{n+1}\|_{0}^{2}\right)
\\
+c\varepsilon^{2}T(F(M_{n}))^{2}\Bigg((1+T)(1+N_{n})(1+M_{n})^{s+1}\Bigg).
\end{multline}
Taking the supremum with respect to $t$ (which does not change the right-hand side) and rearranging,
we have
\begin{multline}\label{toBeAddedTo}
\frac{3}{8}\left(\sup_{t\in[0,T]}\|\partial^{\alpha}\mu^{n+1}\|_{0}^{2}\right)
\leq\frac{1}{2}\|\partial^{\alpha}\mu^{n+1}(0,\cdot)\|_{0}^{2}
\\
+c\varepsilon^{2}T(F(M_{n}))^{2}\Bigg((1+T)(1+N_{n})(1+M_{n})^{s+1}\Bigg).
\end{multline}
We next isolate the time integral on the left-hand side of \eqref{muN+1Second}, again using this with
\eqref{firstSumConclusion}.  As before, we find
\begin{multline}\nonumber
\int_{0}^{t}\|D\partial^{\alpha}\mu^{n+1}\|_{0}^{2}\ d\tau 
\leq \frac{1}{2}\|\partial^{\alpha}\mu^{n+1}(0,\cdot)\|_{0}^{2}
+\frac{1}{8}\left(\sup_{t\in[0,T]}\|\partial^{\alpha}\mu^{n+1}\|_{0}^{2}\right)
\\
+c\varepsilon^{2}T(F(M_{n}))^{2}\Bigg((1+T)(1+N_{n})(1+M_{n})^{s+1}\Bigg).
\end{multline}
We again take the supremum in time, and this again does not affect the right-hand side.  We add the result to 
\eqref{toBeAddedTo}, and rearrange to find the following:
\begin{multline}\label{toBeAddedToAgain}
\frac{1}{4}\left(\sup_{t\in[0,T]}\|\partial^{\alpha}\mu^{n+1}\|_{0}^{2}\right)
+\int_{0}^{T}\|D\partial^{\alpha}\mu^{n+1}\|_{0}^{2}\ d\tau
\\
\leq
\|\partial^{\alpha}\mu^{n+1}(0,\cdot)\|_{0}^{2}
+c\varepsilon^{2}T(F(M_{n}))^{2}\Bigg((1+T)(1+N_{n})(1+M_{n})^{s+1}\Bigg).
\end{multline}

We perform the same manipulations regarding \eqref{ddtUN+1first} and \eqref{secondSumConclusion},
and add the results to \eqref{toBeAddedToAgain}.  These considerations imply the following:
\begin{multline}\label{almostInduction}
\frac{1}{4}\left(\sup_{t\in[0,T]}\|\partial^{\alpha}\mu^{n+1}\|_{0}^{2}
+\sup_{t\in[0,T]}\|\partial^{\alpha}\partial_{x_{j}}w^{n+1}\|_{0}^{2}\right)
\\
+\int_{0}^{T}\|D\partial^{\alpha}\mu^{n+1}(\tau,\cdot)\|_{0}^{2}\ d\tau
+\int_{0}^{T}\|D\partial^{\alpha}\partial_{x_{j}}w^{n+1}(\tau,\cdot)\|_{0}^{2}\ d\tau
\\
\leq \|\partial^{\alpha}\mu^{n+1}(0,\cdot)\|_{0}^{2}
+\|\partial^{\alpha}\partial_{x_{j}}w^{n+1}(T,\cdot)\|_{0}^{2}
\\
+c\varepsilon^{2}T(F(M_{n}))^{2}\Bigg((1+T)(1+N_{n})(1+M_{n})^{s+1}\Bigg).
\end{multline}

We sum \eqref{almostInduction} over multi-indices $\alpha$ such that $0\leq|\alpha|\leq s-1$ and also over
natural numbers $j$ such that $1\leq j\leq n,$ and we multiply by $4;$ this results in the following:
\begin{multline}\nonumber
M_{n+1}+4N_{n+1}\leq 4\|\mu^{n+1}(0,\cdot)\|_{s-1}^{2}+4\|Dw^{n+1}(T,\cdot)\|_{s-1}^{2}
\\
+c\varepsilon^{2}T(F(M_{n}))^{2}\Bigg((1+T)(1+N_{n})(1+M_{n})^{s+1}\Bigg).
\end{multline}
We substitute the boundary conditions \eqref{muN+1BC} and \eqref{uN+1BC}:
\begin{multline}\label{iterativeBound}
M_{n+1}+4N_{n+1}\leq 4\|\mathbb{P}_{\delta}\mu_{0}\|_{s-1}^{2}+4\|D\mathbb{P}_{\delta}w_{T}\|_{s-1}^{2}
\\
+c\varepsilon^{2}T(F(M_{n}))^{2}\Bigg((1+T)(1+N_{n})(1+M_{n})^{s+1}\Bigg).
\end{multline}

We are now ready to both state our smallness constraint and make our inductive hypothesis.  Let 
$\mathcal{S}\in\mathbb{R}$ satisfy
\begin{equation}\nonumber
4\|\mu_{0}\|_{s-1}^{2}+4\|Dw_{T}\|_{s-1}^{2} \leq \mathcal{S}.
\end{equation}
Note that because of the definition of the smoothing operator $\mathbb{P}_{\delta}$ and Plancherel's theorem, an
immediate consequence is
\begin{equation}\nonumber
4\|\mathbb{P}_{\delta}\mu_{0}\|_{s-1}^{2}+4\|D\mathbb{P}_{\delta}w_{T}\|_{s-1}^{2} 
\leq4\|\mu_{0}\|_{s-1}^{2}+4\|Dw_{T}\|_{s-1}^{2}
\leq\mathcal{S},\qquad \forall\delta>0.
\end{equation}
Our smallness assumption is:\\

\noindent
{\bf (H2)} The function $F$ and the constants $c,$ $\varepsilon,$ $T,$ and $\mathcal{S}$ satisfy
\begin{equation}\nonumber
c\varepsilon^{2}T(F(2\mathcal{S}))^{2}\Bigg((1+T)(1+2\mathcal{S})^{s+2}\Bigg)\leq\mathcal{S}.
\end{equation}

Then, our inductive hypothesis is that 
\begin{equation}\label{inductiveHypothesis}
M_{n}+4N_{n}\leq 2\mathcal{S}.
\end{equation}
It is trivial to see that when $n=0,$ since $\mu^{0}=w^{0}=0,$ that $M_{0}+4N_{0}\leq 2\mathcal{S}.$
We assume the inductive hypothesis for some $n\in\mathbb{N}.$  Then, we see that 
$M_{n}\leq 2\mathcal{S}$ and $N_{n}\leq 2\mathcal{S}$ as well.  Since the function $F$ is monotonic, 
we have $F(M_{n})\leq F(2\mathcal{S}).$  Combining the inductive hypothesis \eqref{inductiveHypothesis}
with {\bf(H2)} and the bound
\eqref{iterativeBound}, we conclude $M_{n+1}+4N_{n+1}\leq 2\mathcal{S}.$  Thus, we have proved
that for all $n\in\mathbb{N},$ \eqref{inductiveHypothesis} holds.

The estimate \eqref{inductiveHypothesis}, together with the definition of $M_{n},$ implies that
the sequence $(\mu^{n},w^{n})$ is bounded in the space
$C([0,T];H^{s-1}\times H^{s}),$ uniformly with respect to $n.$  
Our specification of $s$ is sufficiently large so that inspection of \eqref{muN}, \eqref{uN}
shows that $\mu^{n}_{t}$ and $w^{n}_{t}$ are uniformly bounded.  From this, we are able to conclude that
$(\mu^{n},w^{n})$ forms an equicontinuous family, with compact domain $[0,T]\times \mathbb{T}^{d}.$
Applying the Arzela-Ascoli theorem, we find that a subsequence converges uniformly to a limit 
$(\mu,w)\in (C([0,T]\times\mathbb{T}^{d}))^{2}.$  This implies, since the domain is compact, that the convergence
also holds in $C([0,T];H^{0}\times H^{0}).$  Applying Lemma \ref{elementaryInterpolation}, using the 
uniform bound in $H^{s+1}\times H^{s},$ we also find convergence in the space $C([0,T];H^{s'-1}\times H^{s'}),$
for any $s'\in[0,s).$

The uniform bound \eqref{inductiveHypothesis} also implies that the sequence 
$w^{n}$ is uniformly bounded in $L^{2}([0,T];H^{s+1})$
and the sequence $\mu^{n}$ is uniformly bounded in $L^{2}([0,T];H^{s}).$  These are Hilbert spaces and therefore 
bounded sequences have subsequences with weak limits, with the weak limits obeying the same bounds.  So, we may take
$w\in L^{2}([0,T];H^{s+1})$ and $\mu\in L^{2}([0,T];H^{s}),$ with the bound
\begin{equation}\label{limitGainParabolic}
\int_{0}^{T}\|w(t,\cdot)\|_{H^{s+1}}^{2}+\|\mu(t,\cdot)\|_{H^{s}}\ dt \leq c\mathcal{S},
\end{equation}
with $c$ being an absolute constant related to the definitions of the norms.

Integrating \eqref{muN}, \eqref{uN} in time, and using the boundary conditions \eqref{muN+1BC} and \eqref{uN+1BC},
we see that $\mu^{n+1}$ and $w^{n+1}$ satisfy the equations
\begin{equation}\label{muN+1Integrated}
\mu^{n+1}(t,\cdot)=\mathbb{P}_{\delta}\mu_{0}+\int_{0}^{t}\left[\Delta\mu^{n+1}(\tau,\cdot)
-\varepsilon\mathrm{div}\left(\left(\bar{m}+\mu^{n}(\tau,\cdot)\right)\Theta_{p}(\tau,\cdot,\mu^{n},Dw^{n})\right)
\right]\ d\tau,
\end{equation}
\begin{equation}\label{uN+1Integrated}
w^{n+1}(t,\cdot)=\mathbb{P}_{\delta}w_{T}+\int_{t}^{T}\left[
\Delta w^{n+1}(\tau,\cdot)+\varepsilon\Theta(\tau,\cdot,\mu^{n},Dw^{n})\right]
\ d\tau.
\end{equation}
We have established sufficient regularity of the solution $(\mu,w)$ to pass to the limit as $n\rightarrow\infty$ and as
$\delta\rightarrow 0$
in \eqref{muN+1Integrated} and \eqref{uN+1Integrated}; taking these limits, we have
\begin{equation}\label{muLimitIntegrated}
\mu(t,\cdot)=\mu_{0}+\int_{0}^{t}\left[\Delta\mu(\tau,\cdot)-\varepsilon\mathrm{div}
\left(\left(\bar{m}+\mu(\tau,\cdot)\right)\Theta_{p}(\tau,\cdot,\mu,Dw)\right)\right]\ d\tau,
\end{equation}
\begin{equation}\label{uLimitIntegrated}
w(t,\cdot)=w_{T}+\int_{t}^{T}\left[\Delta w(\tau,\cdot)+\varepsilon P\Theta(\tau,\cdot,\mu,Dw)\right]
\ d\tau.
\end{equation}
Differentiating \eqref{muLimitIntegrated} and \eqref{uLimitIntegrated} with respect to time, we find that $(\mu, w)$
are classical solutions of \eqref{uMu1} and \eqref{uMu2}, with the boundary values $\mu_{0}$ and
$w_{T}.$  

To recover $m$ from $\mu,$ one simply needs to add $\bar{m}.$  To recover $u$ from $w$ and $\mu,$ one simply integrates
\eqref{uEquation} with respect to $t,$ since the right-hand side is determined in terms of $w$ and $m.$

This completes the proof.  
We conclude with a more formal statement of what we have proved.
\begin{theorem} \label{existenceTheorem}
Let $T>0$ and $\varepsilon>0$ be given.  
Let $s\geq \left\lceil\frac{d+5}{2}\right\rceil$ and let $\mu_{0}\in H^{s-1}(\mathbb{T}^{d})$ be such that $\bar{m}+\mu_{0}$ is a 
probability measure.  Let $u_{T}\in H^{s}(\mathbb{T}^{d})$ be given.  Assume that the conditions
{\bf(H1)} and {\bf(H2)} are satisfied.  Then there exists $\mu\in L^{\infty}([0,T];H^{s-1})\cap L^{2}([0,T];H^{s})$ and
there exists $u\in L^{\infty}([0,T];H^{s})\cap L^{2}([0,T];H^{s+1})$ such that $\bar{m}+\mu$ is a probability measure for all
$t\in[0,T],$ and such that $(u,\bar{m}+\mu)$ is a classical solution of 
\eqref{uEquation}, \eqref{mEquation}, \eqref{planningBC}.
Furthermore, for all $s'\in[0,s),$ we have $\mu\in C([0,T];H^{s'-1})$ and $u\in C([0,T];H^{s'}).$
\end{theorem}

\begin{remark} \label{existenceRemark}
In \cite{ambroseMFG2}, we gave two existence theorems for non-separable mean field games
with data in the Wiener algebra, with each of these theorems having a different smallness constraint.  Here, we treat
up to three different sources of smallness in a unified constraint.  Clearly, either by taking $\varepsilon$ sufficiently small
for fixed $T$ and $\mathcal{S},$  or by instead taking $T$ sufficiently small for fixed
$\varepsilon$ and $\mathcal{S},$ we may satisfy {\bf(H2)}.  The third source of smallness depends on the form of the
Hamiltonian; if, for instance 
$\mathcal{H}(t,x,m,Du)=m|Du|^{4},$ then the function $F$ could go to zero with $\mathcal{S},$ and
by taking $\mathcal{S}$ sufficiently small, with fixed values of $\varepsilon$ and $T,$ the condition ${\bf(H2)}$ would be
satisfied.  For other choices of the Hamiltonian, however, it may not be the case that $F$ vanishes as $\mathcal{S}$ vanishes.
In summary, this unified condition treats the size of the time horizon, the strength of the coupling in the model, 
and in some cases, the size of the data.  
\end{remark}

\section{Uniqueness}\label{uniquenessSection}

We consider two solutions, $(w^{1},\mu^{1})$ and $(w^{2},\mu^{2})$ in $H^{s}\times H^{s-1},$ for $s>2+\frac{d}{2},$
with the norm of these solutions bounded in these spaces by some $K>0.$
We define $E(t)=E_{\mu}(t)+E_{w}(t),$ with 
\begin{equation}\nonumber
E_{\mu}(t)=\frac{1}{2}\int_{\mathbb{T}^{d}}\left(\mu^{1}(t,x)-\mu^{2}(t,x)\right)^{2}\ dx,
\end{equation}
\begin{equation}\nonumber
E_{w}(t)=\frac{1}{2}\sum_{i=1}^{d}\int_{\mathbb{T}^{d}}\left(\partial_{x_{i}}w^{1}(t,x)-\partial_{x_{i}}w^{2}(t,x)\right)^{2}\ dx.
\end{equation}
Thus, we are measuring the difference of $Dw$ in $L^{2}$ and the difference of $\mu$ also in $L^{2}.$

We must have a Lipschitz property for the Hamiltonian for our uniqueness argument.
We make the following assumption:\\

\noindent{\bf (H3)} For all multi-indices $\beta$ (as described in the beginning of Section 3) with $0\leq |\beta| \leq 2,$
for any $(p^{i},q^{i})$ in a bounded subset of $\mathbb{R}^{d+1},$ there exists a constant $c>0$ such that
\begin{equation}\nonumber
\left|\partial^{\beta}\Theta(t,x,p^{1},q^{1})-\partial^{\beta}\Theta(t,x,p^{2},q^{2})\right|
\leq c \left(|p^{1}-p^{2}| + \sum_{i=1}^{d}|q^{1}_{i}-q^{2}_{i}|\right).
\end{equation}

We are now able to state our uniqueness theorem
\begin{theorem}\label{uniquenessTheorem}
 Let $(u^{1},\bar{m}+\mu^{1})$ and $(u^2,\bar{m}+\mu^{2})$ be two classical solutions of 
\eqref{uEquation}, \eqref{mEquation}, \eqref{planningBC}, with the same data:
\begin{equation}\nonumber
m^{1}(0,\cdot)=m^{2}(0,\cdot),\qquad u^{1}(T,\cdot)=u^{2}(T,\cdot).
\end{equation}
Assume that there exists $K$ such that the solutions are each bounded by $K:$
\begin{equation}\nonumber
\|Du^{i}\|_{H^{s-1}}+\|\mu^{i}\|_{H^{s-1}} \leq K,\qquad i\in\{1,2\},
\end{equation}
for some $s>2+\frac{d}{2}.$  Assume {\bf(H3)} holds.  There exists a nondecreasing 
function $\mathcal{G}:[0,\infty)\rightarrow[0,\infty)$ such that if 
$\varepsilon T (\mathcal{G}(K))<1,$ then $(u^{1},\mu^{1})=(u^{2},\mu^{2}).$
\end{theorem}

\begin{proof}
To estimate the growth of the difference of the two solutions, we take the time derivative of $E,$ starting with $E_{\mu}.$
To begin, we have simply
\begin{equation}\nonumber
\frac{dE_{\mu}}{dt}=\int_{\mathbb{T}^{d}}(\mu^{1}-\mu^{2})(\mu^{1}_{t}-\mu^{2}_{t})\ dx.
\end{equation}
We substitute for $\mu^{1}_{t}$ and $\mu^{2}_{t}$ from \eqref{uMu2}, and we do some preliminary adding 
and subtracting.  This leads us to the expression
\begin{multline}\nonumber
\frac{dE_{\mu}}{dt}=\int_{\mathbb{T}^{d}}(\mu^{1}-\mu^{2})\Delta(\mu^{1}-\mu^{2})\ dx
\\
-\varepsilon\int_{\mathbb{T}^{d}}(\mu^{1}-\mu^{2})\mathrm{div}\left(
(\mu^{1}-\mu^{2})\Theta_{p}(t,x,\mu^{1},Dw^{1})\right)\ dx
\\
-\varepsilon\int_{\mathbb{T}^{d}}(\mu^{1}-\mu^{2})\mathrm{div}\left(
\mu^{2}\left(\Theta_{p}(t,x,\mu^{1},Dw^{1})-\Theta_{p}(t,x,\mu^{2},Dw^{2})\right)\right)\ dx
\\
-\varepsilon\bar{m}\int_{\mathbb{T}^{d}}(\mu^{1}-\mu^{2})\mathrm{div}\left(
\Theta_{p}(t,x,\mu^{1},Dw^{1})-\Theta_{p}(t,x,\mu^{2},Dw^{2})\right)\ dx.
\end{multline}
We apply the divergence operators on the right-hand side, making the expansion
\begin{equation}\nonumber
\frac{dE_{\mu}}{dt}=\sum_{\ell=1}^{14}V_{\ell},
\end{equation}
where we now list each of these terms:
\begin{equation}\nonumber
V_{1}=\int_{\mathbb{T}^{d}}(\mu^{1}-\mu^{2})\Delta(\mu^{1}-\mu^{2})\ dx,
\end{equation}
\begin{equation}\nonumber
V_{2}=-\varepsilon\int_{\mathbb{T}^{d}}(\mu^{1}-\mu^{2})\left(\nabla\left(\mu^{1}-\mu^{2}\right)\right)\cdot
\Theta_{p}(t,x,\mu^{1},Dw^{1})\ dx,
\end{equation}
\begin{equation}\nonumber
V_{3}=-\varepsilon\int_{\mathbb{T}^{d}}\left(\mu^{1}-\mu^{2}\right)^{2}
\mathrm{div}\left(\Theta_{p}(t,x,\mu^{1},Dw^{1})\right)\ dx,
\end{equation}
\begin{equation}\nonumber
V_{4}=-\varepsilon\int_{\mathbb{T}^{d}}(\mu^{1}-\mu^{2})(\nabla\mu^{2})\cdot
\left(\Theta_{p}(t,x,\mu^{1},Dw^{1})-\Theta_{p}(t,x,\mu^{2},Dw^{2})\right)\ dx,
\end{equation}
\begin{equation}\nonumber
V_{5}=-\varepsilon\int_{\mathbb{T}^{d}}(\mu^{1}-\mu^{2})(\mu^{2})\sum_{i=1}^{d}\left[
\Theta_{p_{i}x_{i}}(t,x,\mu^{1},Dw^{1})-\Theta_{p_{i}x_{i}}(t,x,\mu^{2},Dw^{2})\right]\ dx,
\end{equation}
\begin{equation}\nonumber
V_{6}=-\varepsilon\int_{\mathbb{T}^{d}}(\mu^{1}-\mu^{2})(\mu^{2})\sum_{i=1}^{d}
\left[\Theta_{p_{i}q}(t,x,\mu^{1},Dw^{1})\frac{\partial\mu^{1}}{\partial x_{i}}
-\Theta_{p_{i}q}(t,x,\mu^{2},Dw^{2})\frac{\partial\mu^{1}}{\partial x_{i}}\right]\ dx,
\end{equation}
\begin{equation}\nonumber
V_{7}=-\varepsilon\int_{\mathbb{T}^{d}}(\mu^{1}-\mu^{2})(\mu^{2})\sum_{i=1}^{d}\left[
\Theta_{p_{i}q}(t,x,\mu^{2},Dw^{2})\left(\frac{\partial(\mu^{1}-\mu^{2})}{\partial x_{i}}\right)
\right]\ dx,
\end{equation}
\begin{multline}\nonumber
V_{8}=-\varepsilon\int_{\mathbb{T}^{d}}(\mu^{1}-\mu^{2})(\mu^{2})\sum_{i=1}^{d}\sum_{j=1}^{d}\Bigg[
\Theta_{p_{i}p_{j}}(t,x,\mu^{1},Dw^{1})\frac{\partial^{2}w^{1}}{\partial x_{i}\partial x_{j}}
\\
-\Theta_{p_{i}p_{j}}(t,x,\mu^{2},Dw^{2})\frac{\partial^{2}w^{1}}{\partial x_{i}\partial x_{j}}
\Bigg]\ dx,
\end{multline}
\begin{equation}\nonumber
V_{9}=-\varepsilon\int_{\mathbb{T}^{d}}(\mu^{1}-\mu^{2})(\mu^{2})\sum_{i=1}^{d}\sum_{j=1}^{d}\left[
\Theta_{p_{i}p_{j}}(t,x,\mu^{2},Dw^{2})\left(\frac{\partial^{2}(w^{1}-w^{2})}{\partial x_{i}\partial x_{j}}\right)
\right]\ dx,
\end{equation}
\begin{equation}\nonumber
V_{10}=-\varepsilon\bar{m}\int_{\mathbb{T}^{d}}(\mu^{1}-\mu^{2})\sum_{i=1}^{d}\left[
\Theta_{p_{i}x_{i}}(t,x,\mu^{1},Dw^{1})-\Theta_{p_{i}x_{i}}(t,x,\mu^{2},Dw^{2})\right]\ dx,
\end{equation}
\begin{equation}\nonumber
V_{11}=-\varepsilon\bar{m}\int_{\mathbb{T}^{d}}(\mu^{1}-\mu^{2})\sum_{i=1}^{d}\left[
\Theta_{p_{i}q}(t,x,\mu^{1},Dw^{1})\frac{\partial\mu^{1}}{\partial x_{i}}
-\Theta_{p_{i}q}(t,x,\mu^{2},Dw^{1})\frac{\partial\mu^{1}}{\partial x_{i}}
\right]\ dx,
\end{equation}
\begin{equation}\nonumber
V_{12}=-\varepsilon\bar{m}\int_{\mathbb{T}^{d}}(\mu^{1}-\mu^{2})\sum_{i=1}^{d}
\Theta_{p_{i}q}(t,x,\mu^{2},Dw^{2})
\left(\frac{\partial(\mu^{1}-\mu^{2})}{\partial x_{i}}\right)
\ dx,
\end{equation}
\begin{multline}\nonumber
V_{13}=-\varepsilon\bar{m}\int_{\mathbb{T}^{d}}(\mu^{1}-\mu^{2})\sum_{i=1}^{d}\sum_{j=1}^{d}\Bigg[
\Theta_{p_{i}p_{j}}(t,x,\mu^{1},Dw^{1})\frac{\partial^{2}w^{1}}{\partial x_{i}\partial x_{j}}
\\
-\Theta_{p_{i}p_{j}}(t,x,\mu^{2},Dw^{2})\frac{\partial^{2}w^{1}}{\partial x_{i}\partial x_{j}}
\Bigg]\ dx,
\end{multline}
and finally,
\begin{equation}\nonumber
V_{14}=-\varepsilon\bar{m}\int_{\mathbb{T}^{d}}(\mu^{1}-\mu^{2})\sum_{i=1}^{d}\sum_{j=1}^{d}\left[
\Theta_{p_{i}p_{j}}(t,x,\mu^{2},Dw^{2})
\left(\frac{\partial^{2}(w^{1}-w^{2})}{\partial x_{i}\partial x_{j}}\right)
\right]\ dx.
\end{equation}

We integrate $V_{1}$ by parts:
\begin{equation}\label{V1ByParts}
V_{1}=-\int_{\mathbb{T}^{d}}\left|\nabla\left(\mu^{1}-\mu^{2}\right)\right|^{2}\ dx.
\end{equation}
We also integrate each of $V_{2},$ $V_{7},$ and $V_{12}$ by parts:
\begin{equation}\nonumber
V_{2}=\frac{\varepsilon}{2}\int_{\mathbb{T}^{d}}(\mu^{1}-\mu^{2})^{2}
\mathrm{div}\left(\Theta_{p}(t,x,\mu^{1},Dw^{1})\right)\ dx,
\end{equation}
\begin{equation}\nonumber
V_{7}=\frac{\varepsilon}{2}\sum_{i=1}^{d}\int_{\mathbb{T}^{d}}(\mu^{1}-\mu^{2})^{2}\frac{\partial}{\partial x_{i}}
\left((\mu^{2})\Theta_{p_{i}q}(t,x,\mu^{1},Dw^{1})\right)\ dx,
\end{equation}
\begin{equation}\nonumber
V_{12}=\frac{\varepsilon\bar{m}}{2}\sum_{i=1}^{d}\int_{\mathbb{T}^{d}}(\mu^{1}-\mu^{2})^{2}
\frac{\partial}{\partial x_{i}}\left(\Theta_{p_{i}q}(t,x,\mu^{1},Dw^{1})\right)\ dx.
\end{equation}
These terms, and also $V_{3},$ are then bounded in terms of the energy, using the bound on the solutions.
We distinguish between two kinds of bounds, though; there exists a nondecreasing function $\mathcal{G}_{1}$
which may be taken so as to converge to zero if $K$ vanishes, such that
\begin{equation}\label{V-7}
V_{7}\leq \varepsilon \mathcal{G}_{1}(K) E_{\mu}.
\end{equation}
The fact that $\mathcal{G}_{1}$ can be taken to vanish with $K$ is because of the presence of the linear factor $\mu^{2}$
in $V_{7}.$  Also note that the regularity requirement $s>2+\frac{d}{2}$ allowed us here to estimate $\mu^{2}$ in $L^{\infty};$
the requirement comes into play in the same way several times throughout the rest of the argument.
On the other hand, we have a nondecreasing function $\mathcal{G}_{2}$ such that
\begin{equation}\label{V-2-3-and-12}
V_{2}+V_{3}+V_{12}\leq \varepsilon \mathcal{G}_{2}(K) E_{\mu}.
\end{equation}

For most of the remaining terms, we estimate them using the Lipschitz properties of $\Theta_{p}$ and its derivatives;
these terms satisfy
\begin{equation}\label{V-many}
V_{4}+V_{5}+V_{6}+V_{8}+V_{11}+V_{13}
\leq \varepsilon \mathcal{G}_{1}(K)(E_{\mu}+E_{\mu}^{1/2}E_{w}^{1/2}),
\end{equation}
where $\mathcal{G}_{1}(K)$ is as before, and where its
vanishing property is again because of the presence of linear factors such as $\mu^{2}$ in the terms.
Another term relies on the Lipschitz estimate for $\Theta_{p},$ but does not have such a linear factor of the unknowns present;
for this, we again have the existence of $\mathcal{G}_{2}(K)$ such that
\begin{equation}\label{V-10}
V_{10}\leq \varepsilon \mathcal{G}_{2}(K)(E_{\mu}+E_{\mu}^{1/2}E_{w}^{1/2}).
\end{equation}

This leaves two more terms to deal with, $V_{9}$ and $V_{14}.$  We will use Young's inequality for these, and 
later bound them by a contribution from $E_{w}.$
For $V_{9},$ we begin by bounding $\Theta_{p_{i}p_{j}}$ and $\mu^{2}$ with $\mathcal{G}_{1}(K):$
\begin{equation}\nonumber
V_{9}\leq\varepsilon \mathcal{G}_{1}(K)\sum_{i=1}^{d}\sum_{j=1}^{d}\int_{\mathbb{T}^{d}}(\mu^{1}-\mu^{2})
\frac{\partial^{2}(w^{1}-w^{2})}{\partial_{x_{i}}\partial_{x_{j}}}\ dx.
\end{equation}
We then apply Young's inequality, with parameter $4\varepsilon \mathcal{G}_{1}(K):$
\begin{multline}\nonumber
V_{9}\leq\varepsilon^{2} \mathcal{G}_{1}(K)\int_{\mathbb{T}^{d}}(\mu^{1}-\mu^{2})^{2}\ dx
+\frac{1}{8}\sum_{j=1}^{d}\int_{\mathbb{T}^{d}}(\partial_{x_{j}}(Dw^{1}-Dw^{2}))^{2}\ dx
\\
\leq \varepsilon^{2}\mathcal{G}_{1}(K)E_{\mu}
+\frac{1}{8}\sum_{j=1}^{d}\int_{\mathbb{T}^{d}}(\partial_{x_{j}}(Dw^{1}-Dw^{2}))^{2}\ dx.
\end{multline}
The remaining term, $V_{14},$ is entirely similar, except that we use $\mathcal{G}_{2}$ instead of $\mathcal{G}_{1}:$
\begin{equation}\nonumber
V_{14}\leq 
\varepsilon^{2}\mathcal{G}_{2}(K)E_{\mu}
+\frac{1}{8}\sum_{j=1}^{d}\int_{\mathbb{T}^{d}}(\partial_{x_{j}}(Dw^{1}-Dw^{2}))^{2}\ dx.
\end{equation}
Adding these results for $V_{9}$ and $V_{14},$ we have
\begin{equation}\label{V-9-and-14}
V_{9}+V_{14}\leq \varepsilon(\mathcal{G}_{1}(K)+\mathcal{G}_{2}(K))E_{\mu}
+\frac{1}{4}\sum_{j=1}^{d}\int_{\mathbb{T}^{d}}(\partial_{x_{j}}(Dw^{1}-Dw^{2}))^{2}\ dx.
\end{equation}

Summarizing the terms we have estimated so far, leaving $V_{1}$ out for the moment, 
by adding the bounds \eqref{V-7}, \eqref{V-2-3-and-12}, \eqref{V-many}, \eqref{V-10}, and
\eqref{V-9-and-14}, 
we have concluded the 
following:
\begin{equation}\label{sumOfMostVTerms}
\sum_{\ell=2}^{14}V_{\ell}\leq \varepsilon(\mathcal{G}_{1}(K)+\mathcal{G}_{2}(K))(E_{w}+E_{\mu}) + 
\frac{1}{4}\sum_{j=1}^{d}\int_{\mathbb{T}^{d}}(\partial_{x_{j}}^{2}(Dw^{1}-Dw^{2}))^{2}\ dx.
\end{equation}

We turn our attention to $E_{w},$ and we write 
$E_{w}=\sum_{j=1}^{d}E^{j}_{w},$ 
with
\begin{equation}\nonumber
\frac{dE^{j}_{w}}{dt}=\int_{\mathbb{T}^{d}}(\partial_{x_{j}}w^{1}-\partial_{x_{j}}w^{2})
\partial_{t}(\partial_{x_{j}}w^{1}-\partial_{x_{j}}w^{2})\ dx.
\end{equation}
We then add and subtract to make the following decomposition:
\begin{equation}\nonumber
\frac{dE^{j}_{w}}{dt}=\sum_{\ell=1}^{6}W^{j}_{\ell},
\end{equation}
with
\begin{equation}\nonumber
W^{j}_{1}=-\int_{\mathbb{T}^{d}}(\partial_{x_{j}}w^{1}-\partial_{x_{j}}w^{2})
\Delta(\partial_{x_{j}}w^{1}-\partial_{x_{j}}w^{2})\ dx,
\end{equation}
\begin{equation}\nonumber
W^{j}_{2}=-\varepsilon\int_{\mathbb{T}^{d}}(\partial_{x_{j}}w^{1}-\partial_{x_{j}}w_{2})
\left(\Theta_{x_{j}}(t,x,\mu^{1},Dw^{1})-\Theta_{x_{j}}(t,x,\mu^{2},Dw^{2})\right)\ dx,
\end{equation}
\begin{equation}\nonumber
W^{j}_{3}=-\varepsilon\int_{\mathbb{T}^{d}}(\partial_{x_{j}}w^{1}-\partial_{x_{j}}w^{2})
\left(\Theta_{q}(t,x,\mu^{1},Dw^{1})\mu^{1}_{x_{j}}-\Theta_{q}(t,x,\mu^{2},Dw^{2})\mu^{1}_{x_{j}}\right)\ dx,
\end{equation}
\begin{equation}\nonumber
W^{j}_{4}=-\varepsilon\int_{\mathbb{T}^{d}}(\partial_{x_{j}}w^{1}-\partial_{x_{j}}w^{2})
\left(\Theta_{p}(t,x,\mu^{2},Dw^{2})\mu^{1}_{x_{j}}-\Theta_{q}(t,x,\mu^{2},Dw^{2})\mu^{2}_{x_{j}}\right)\ dx,
\end{equation}
\begin{equation}\nonumber
W^{j}_{5}=-\varepsilon\sum_{i=1}^{d}\int_{\mathbb{T}^{d}}(\partial_{x_{j}}w^{1}-\partial_{x_{j}}w^{2})
\left(\Theta_{p_{i}}(t,x,\mu^{1},Dw^{1})\partial^{2}_{x_{i}x_{j}}w^{1}
-\Theta_{p_{i}}(t,x,\mu^{2},Dw^{2})\partial^{2}_{x_{i}x_{j}}w^{1}\right)\ dx,
\end{equation}
\begin{equation}\nonumber
W^{j}_{6}=-\varepsilon\sum_{i=1}^{d}\int_{\mathbb{T}^{d}}(\partial_{x_{j}}w^{1}-\partial_{x_{j}}w^{2})
\left(\Theta_{p_{i}}(t,x,\mu^{2},Dw^{2})\partial^{2}_{x_{i}x_{j}}w^{1}
-\Theta_{p_{i}}(t,x,\mu^{2},Dw^{2})\partial^{2}_{x_{i}x_{j}}w^{2}\right)\ dx.
\end{equation}

We integrate $W^{j}_{1}$ by parts:
\begin{equation}\nonumber
W^{j}_{1}=\int_{\mathbb{T}^{d}}\left|\nabla\partial_{x_{j}}\left(w^{1}-w^{2}\right)\right|^{2}\ dx.
\end{equation}

We may also integrate $W_{6}^{j}$ by parts, to find the following:
\begin{equation}\nonumber
W_{6}^{j}=\frac{\varepsilon}{2}\sum_{i=1}^{d}\int_{\mathbb{T}^{d}}
(\partial_{x_{j}}w^{1}-\partial_{x_{j}}w^{2})^{2}\partial_{x_{i}}\left(\Theta_{p_{i}}(t,x,\mu^{2},Dw^{2})\right)\ dx.
\end{equation}

With $\mathcal{G}_{1}$ and $\mathcal{G}_{2}$ as before, we may estimate many of the terms forthwith:
\begin{equation}\label{W-2-and-6}
W_{2}^{j}+W_{6}^{j}
\leq
\varepsilon \mathcal{G}_{2}(K)(E_{w}+E_{w}^{1/2}E_{\mu}^{1/2}),
\end{equation}
\begin{equation}\label{W-3-and-5}
W_{3}^{j}+W_{5}^{j}
\leq \varepsilon \mathcal{G}_{1}(K)(E_{w}+E_{w}^{1/2}E_{\mu}^{1/2}).
\end{equation}

We are left with one term to deal with more carefully, $W_{4}^{j}.$    We first bound $\Theta_{p}(t,x,\mu^{2},Dw^{2})$
in $L^{\infty}$ by $\mathcal{G}_{2}(K),$ finding
\begin{equation}\nonumber
W_{4}^{j}\leq \varepsilon \mathcal{G}_{2}(K)\int_{\mathbb{T}^{d}}(\partial_{x_{j}}w^{1}-\partial_{x_{j}}w^{2})
\partial_{x_{j}}(\mu^{1}-\mu^{2})\ dx.
\end{equation}
We next use Young's inequality, with $2\varepsilon \mathcal{G}_{2}(K)$ as the parameter, finding
\begin{equation}\nonumber
W_{4}^{j}\leq
\varepsilon^{2} \mathcal{G}_{2}(K)E_{w} + \frac{1}{4}\int_{\mathbb{T}^{d}}(\partial_{x_{j}}\mu^{1}-\partial_{x_{j}}\mu^{2})^{2}\ dx.
\end{equation}

Adding contributions from \eqref{W-2-and-6} and \eqref{W-3-and-5} to this,
we have
\begin{equation}\label{sumOfMostWTerms+V1}
\sum_{\ell=2}^{6}\sum_{j=1}^{d}W_{\ell}^{j}
\leq \varepsilon(\mathcal{G}_{1}(K)+\mathcal{G}_{2}(K))(E_{w}+E_{\mu})
+\frac{1}{4}\int_{\mathbb{T}^{d}}|D\mu^{1}-D\mu^{2}|^{2}\ dx.
\end{equation}

We are now ready to integrate with respect to time.  Integrating $\frac{dE_{\mu}}{dt}$ over the interval $[0,t],$ we find
\begin{equation}\nonumber
E_{\mu}(t)=E_{\mu}(0)+\int_{0}^{t}\sum_{\ell=1}^{14}V_{\ell}.
\end{equation}
We then use the bound \eqref{sumOfMostVTerms}, finding
\begin{multline}\nonumber
E_{\mu}(t)\leq E_{\mu}(0)+\varepsilon T (\mathcal{G}_{1}(K)+\mathcal{G}_{2}(K))(E_{\mu}+E_{w})
\\
+\int_{0}^{t}\left[V_{1}+\frac{1}{4}\sum_{j=1}^{d}\int_{\mathbb{T}^{d}}(\partial_{x_{j}}(Dw^{1}-Dw^{2}))^{2}\ dx\right]\ d\tau.
\end{multline}
We next integrate $\frac{dE_{w}}{dt}$ over the interval $[t,T],$ finding
\begin{equation}\nonumber
E_{w}(t)=E_{w}(T)-\int_{t}^{T}\sum_{j=1}^{d}\sum_{\ell=1}^{6}W^{j}_{\ell}\ d\tau.
\end{equation}
We use the estimate \eqref{sumOfMostWTerms+V1}, then, as follows:
\begin{multline}\nonumber
E_{w}(t)\leq E_{w}(T)
+\varepsilon T(\mathcal{G}_{1}(K)+\mathcal{G}_{2}(K))(E_{w}+E_{\mu})
-\sum_{j=1}^{d}\int_{t}^{T}W_{1}^{j}\ d\tau
\\
+\frac{1}{4}\int_{t}^{T}\int_{\mathbb{T}^{d}}|D\mu^{1}-D\mu^{2}|^{2}\ dxd\tau.
\end{multline}
We apply the definitions of $V_{1}$ and $W^{j}_{1},$ and summarize what we have found thus far:
\begin{multline}\label{almostDoneUniqueness}
E_{w}(t)+E_{\mu}(t)+\int_{0}^{t}\|D\mu^{1}-D\mu^{2}\|_{0}^{2}\ d\tau+\int_{t}^{T}\|D^{2}w^{1}-D^{2}w^{2}\|_{0}^{2}\ d\tau
\\
\leq
E_{w}(T)+E_{\mu}(0)+\varepsilon T(\mathcal{G}_{1}(K)+\mathcal{G}_{2}(K))(E_{w}(t)+E_{\mu}(t))
\\
+\frac{1}{4}\int_{t}^{T}\|D\mu^{1}-D\mu^{2}\|_{0}^{2}\ d\tau
+\frac{1}{4}\int_{0}^{t}\|D^{2}w^{1}-D^{2}w^{2}\|_{0}^{2}\ d\tau.
\end{multline}
On the right-hand side, we bound the integrals with integrals over the entire time interval:
\begin{multline}\label{almostDoneUniqueness2}
E_{w}(t)+E_{\mu}(t)+\int_{0}^{t}\|D\mu^{1}-D\mu^{2}\|_{0}^{2}\ d\tau+\int_{t}^{T}\|D^{2}w^{1}-D^{2}w^{2}\|_{0}^{2}\ d\tau
\\
\leq
E_{w}(T)+E_{\mu}(0)+\varepsilon T(\mathcal{G}_{1}(K)+\mathcal{G}_{2}(K))(E_{w}(t)+E_{\mu}(t))
\\
+\frac{1}{4}\int_{0}^{T}\|D\mu^{1}-D\mu^{2}\|_{0}^{2}\ d\tau
+\frac{1}{4}\int_{0}^{T}\|D^{2}w^{1}-D^{2}w^{2}\|_{0}^{2}\ d\tau.
\end{multline}
Isolating the first integral on the left-hand side of \eqref{almostDoneUniqueness2}, we conclude from this that
\begin{multline}\label{workaroundParabolic1}
\int_{0}^{t}\|D\mu^{1}-D\mu^{2}\|_{0}^{2}\ d\tau
\leq
E_{w}(T)+E_{\mu}(0)+\varepsilon T(\mathcal{G}_{1}(K)+\mathcal{G}_{2}(K))(E_{w}(t)+E_{\mu}(t))
\\
+\frac{1}{4}\int_{0}^{T}\|D\mu^{1}-D\mu^{2}\|_{0}^{2}\ d\tau
+\frac{1}{4}\int_{0}^{T}\|D^{2}w^{1}-D^{2}w^{2}\|_{0}^{2}\ d\tau.
\end{multline}
We take the supremum over time in both sides of \eqref{workaroundParabolic1}, finding
\begin{multline}\label{workaroundParabolic2}
\int_{0}^{T}\|D\mu^{1}-D\mu^{2}\|_{0}^{2}\ d\tau
\leq
E_{w}(T)+E_{\mu}(0)
\\
+\varepsilon T(\mathcal{G}_{1}(K)+\mathcal{G}_{2}(K))\left(\sup_{t\in[0,T]}(E_{w}(t)+E_{\mu}(t))\right)
\\
+\frac{1}{4}\int_{0}^{T}\|D\mu^{1}-D\mu^{2}\|_{0}^{2}\ d\tau
+\frac{1}{4}\int_{0}^{T}\|D^{2}w^{1}-D^{2}w^{2}\|_{0}^{2}\ d\tau.
\end{multline}
We similarly treat the other integral on the left-hand side of \eqref{almostDoneUniqueness2}, finding
\begin{multline}\label{workaroundParabolic3}
\int_{0}^{T}\|D^{2}w^{1}-D^{2}w^{2}\|_{0}^{2}\ d\tau
\leq
E_{w}(T)+E_{\mu}(0)
\\
+\varepsilon T(\mathcal{G}_{1}(K)+\mathcal{G}_{2}(K))\left(\sup_{t\in[0,T]}(E_{w}(t)+E_{\mu}(t))\right)
\\
+\frac{1}{4}\int_{0}^{T}\|D\mu^{1}-D\mu^{2}\|_{0}^{2}\ d\tau
+\frac{1}{4}\int_{0}^{T}\|D^{2}w^{1}-D^{2}w^{2}\|_{0}^{2}\ d\tau.
\end{multline}
We add \eqref{workaroundParabolic2} and \eqref{workaroundParabolic3} and rearrange the integrals, finding
\begin{multline}\nonumber
\int_{0}^{T}\|D\mu^{1}-D\mu^{2}\|_{0}^{2}\ d\tau
+\int_{0}^{T}\|D^{2}w^{1}-D^{2}w^{2}\|_{0}^{2}\ d\tau
\\
\leq 2E_{w}(T)+2E_{\mu}(0)
+\varepsilon T(\mathcal{G}_{1}(K)+\mathcal{G}_{2}(K))\left(\sup_{t\in[0,T]}(E_{w}(t)+E_{\mu}(t))\right).
\end{multline}
Using this with \eqref{almostDoneUniqueness2}, we have
\begin{multline}\label{finalUniqueness}
\sup_{t\in[0,T]}(E_{w}(t)+E_{\mu}(t))\\
\leq \frac{3}{2}E_{w}(T)+\frac{3}{2}E_{\mu}(0)
+\varepsilon T(\mathcal{G}_{1}(K)+\mathcal{G}_{2}(K))\left(\sup_{t\in[0,T]}(E_{w}(t)+E_{\mu}(t))\right).
\end{multline}

If $\varepsilon T (\mathcal{G}_{1}(K)+\mathcal{G}_{2}(K))<1,$ then we have established (local) uniqueness of solutions.
For then, if $w^{1}(T,\cdot)=w^{2}(T,\cdot)$ and if $\mu^{1}(0,\cdot)=\mu^{2}(0,\cdot),$ then
$E_{w}(T)=E_{\mu}(0)=0,$ and \eqref{finalUniqueness} implies
\begin{equation}\nonumber
\sup_{t\in[0,T]}(E_{w}(t)+E_{\mu}(t)) \leq 0.
\end{equation}
Thus $w^{1}=w^{2}$ and $\mu^{1}=\mu^{2}.$ 
This completes the proof.
\end{proof}

\begin{remark} \label{uniquenessRemark}
As we did in Remark \ref{existenceRemark} 
following the statement of our smallness condition for use in Theorem \ref{existenceTheorem},
we remark now on the smallness condition for Theorem \ref{uniquenessTheorem}.  
The function $\mathcal{G}(K)$ mentioned in the statement of Theorem \ref{uniquenessTheorem}
is equal to $\mathcal{G}_{1}(K)+\mathcal{G}_{2}(K).$  Clearly the smallness condition may be satisfied by taking
either $\varepsilon$ or $T$ sufficiently small, and we remarked on this similarly in Remark \ref{existenceRemark} 
for existence.
It is also possible, for some Hamiltonians, that uniqueness follows in the case of small data.  The function
$\mathcal{G}_{1}(K)$ does go to zero as $K$ goes to zero, but whether $\mathcal{G}_{2}(K)$ does as well depends on the Hamiltonian.
For example, for the Hamiltonian $\mathcal{H}=m|Du|^{2},$ we have $\mathcal{H}_{p}=2mDu,$ and the Lipschitz constant
for $\mathcal{H}_{p}$ then does not become small when the solution is small.  For other choices
of Hamiltonian, however, the Lipschitz constant may be small when the solution is, and for such Hamiltonians, 
one would gain from this theorem uniqueness of small solutions for larger values of $\varepsilon$ and $T.$
\end{remark}

\section{Extensions and Discussion}\label{discussionSection}

Having proved our main existence and uniqueness theorems, we now close with some  remarks.  

\subsection{Extension to more general payoff conditions}\label{payoffExtension}

We have carried out the above analysis with a prescribed terminal condition, i.e. with boundary conditions 
\eqref{planningBC}.
As discussed in \cite{cirantNew}, there seem two possible ways to prove both existence and uniqueness for the more general 
conditions \eqref{payoffBC}.  One of these is the assumption pursued in \cite{cirantNew}, in which the payoff function 
$G$ is regularizing, and the other option, mentioned briefly in \cite{cirantNew}, is to make some further smallness assumption.
We state now and give a sketch of the proof (as most of the proof is the same as for Theorem \ref{existenceTheorem} above) of 
existence of solutions under such an additional smallness hypothesis.

To state the additional smallness condition, it will help if we introduce some further notation.
Loosely speaking, we let $\varepsilon_{1}$ be an upper bound for $\Theta_{p_{i}p_{j}};$ this is made precise in 
\eqref{definitionOfVarepsilon1} just below.  
In the following theorem, in which we have a non-regularizing payoff function $G,$ we will assume that the product
$\varepsilon\varepsilon_{1}$ is sufficiently small.  Clearly this could be achieved for Hamiltonians with fairly general forms by 
simply having $\varepsilon$ small.  Alternatively, for some Hamiltonians such as $m^{p}|Du|^{4},$ second derivatives with respect to $p$
can be small if the data for $u$ is small; i.e. for such Hamiltonians, $\varepsilon_{1}$ may itself be taken small.

\begin{theorem} Assume that the payoff function $G$ is a continuous mapping from $\mathbb{T}^{d}\times H^{s}(\mathbb{T}^{d})$ to
$H^{s}(\mathbb{T}^{d}).$
Let $T>0$ and $\varepsilon>0$ be given.  
Let $s\geq \left\lceil\frac{d+5}{2}\right\rceil$ and let $\mu_{0}\in H^{s}(\mathbb{T}^{d})$ be such that $\bar{m}+\mu_{0}$ is a 
probability measure.  Let $u_{T}\in H^{s}(\mathbb{T}^{d})$ be given.  Assume that the condition
{\bf(H1)} is satisfied.   
Let $\varepsilon_{1}>0$ satisfy
\begin{equation}\label{definitionOfVarepsilon1}
\sup_{i,j}\sup_{t,x}|\Theta_{p_{i}p_{j}}(t,x,\mu_{0},DG(x,\mu_{0}+\bar{m})|<\frac{\varepsilon_{1}}{2}.
\end{equation}
If $T$ is sufficiently small and if
the product $\varepsilon\varepsilon_{1}$ is sufficiently small,
then there exists $\mu\in L^{\infty}([0,T];H^{s})\cap L^{2}([0,T];H^{s+1})$ and
there exists $u\in L^{\infty}([0,T];H^{s})\cap L^{2}([0,T];H^{s+1})$ such that $\bar{m}+\mu$ is a probability measure for all
$t\in[0,T],$ and such that $(u,\bar{m}+\mu)$ is a classical solution of 
\eqref{uEquation}, \eqref{mEquation}, \eqref{payoffBC}.
Furthermore, for all $s'\in[0,s),$ we have $\mu\in C([0,T];H^{s'})$ and $u\in C([0,T];H^{s'}).$
\end{theorem}

\begin{proof}
We follow the proof of Theorem \ref{existenceTheorem} with some changes.

Considering the general payoff problem \eqref{payoffBC} instead of the prescribed problem \eqref{planningBC}, 
we replace \eqref{uN+1BC} with the following:
\begin{equation}\nonumber
w^{n+1}(T,x)=\mathbb{P}_{\delta} PG(x,m^{n}(T,x)).
\end{equation}
Of course, there is the related change in the Duhamel formula as well:
\begin{multline}\nonumber
w^{n+1}(t,\cdot)=e^{\Delta(T-t)}\mathbb{P}_{\delta}PG(\cdot,m^{n}(T,\cdot))
\\
-\varepsilon P\int_{t}^{T}e^{\Delta(s-t)}
\Theta(s,\cdot,\mu^{n}(s,\cdot),Dw^{n}(s,\cdot))\ ds.
\end{multline}

Taking $u\in H^{s}$ and $m\in H^{s},$ there is only one problematic term in the analysis, and it is the analogue of the
term $IV_{A}$ in the proof of Theorem \ref{existenceTheorem}.  We define the new norms as
\begin{equation}\nonumber
\widetilde{M}_{n}=\sup_{t\in[0,T]}\left(\|Dw^{n}\|_{s-1}^{2}+\|\mu^{n}\|_{s}^{2}\right),
\end{equation}
\begin{equation}\nonumber
\widetilde{N}_{n}=\sum_{1\leq|\alpha|\leq s}\int_{0}^{T}\left[
\|\partial^{\alpha}Dw^{n}\|_{0}^{2}+\|\partial^{\alpha}D\mu^{n}\|_{0}^{2}\right]\ d\tau,
\end{equation}
in contrast to \eqref{mndef} and \eqref{nndef}.
As in the proof of Theorem \ref{existenceTheorem}, the proof relies on induction, with the inductive hypothesis
\begin{equation}\nonumber
\widetilde{M}_{n}+4\widetilde{N}_{n}\leq 2\mathcal{S},
\end{equation}
where $\mathcal{S}>0$ is taken to satisfy
\begin{equation}\nonumber
4\|\mu_{0}\|_{s}^{2}+4\|Dw_{T}\|_{s-1}^{2}\leq \mathcal{S}.
\end{equation}
Considering \eqref{definitionOfVarepsilon1}, and in light of the inductive hypothesis, we can take $T>0$ to be sufficiently small 
(such that this smallness requirement will turn out to be independent of $n$) such that
\begin{equation}\label{newSmall}
\sup_{i,j}\sup_{t,x}|\Theta_{p_{i}p_{j}}(t,x,\mu^{n},Dw^{n})|<\varepsilon_{1}.
\end{equation}
To be clear, this is because the time derivative of $\mu^{n}$ is bounded, so that $\mu^{n}$ and thus 
$\Theta_{p_{i}p_{j}}$ will not vary too much over a bounded time interval.

All terms except the analogue of $IV_{A}$ are estimated
the same as previously; we call this analogue $\widetilde{IV}_{A}.$
The formula for $\widetilde{IV}_{A}$ is actually just the same as the formula for $IV_{A}$, but with the understanding
that the order of $\alpha$ may now be higher (that is, previously the order of $\alpha$ was at most $s-1$ while it now may be of order up to
$s$).  We have
\begin{multline}\nonumber
\widetilde{IV}_{A}=
\\
-\varepsilon\int_{0}^{t}\int_{\mathbb{T}^{d}}
\left(\partial^{\alpha}\mu^{n+1}\right)
(\mu^{n}+\bar{m})\sum_{i=1}^{d}\sum_{j=1}^{d}
\left[\left(\Theta_{p_{i}p_{j}}(\tau,x,\mu^{n},Dw^{n})\right)
\left(\partial^{\alpha}\partial_{x_{i}x_{j}}^{2}w^{n}\right)\right]\ dx d\tau.
\end{multline}
We integrate by parts, and decompose the result as $\widetilde{IV}_{A}=\widetilde{IV}_{A,i}+\widetilde{IV}_{A,ii},$
with
\begin{equation}\nonumber
\widetilde{IV}_{A,i}=\varepsilon\int_{0}^{t}\int_{\mathbb{T}^{d}}\sum_{i=1}^{d}\sum_{j=1}^{d}(\partial^{\alpha}\partial_{x_{i}}\mu^{n+1})
(\mu^{n}+\bar{m})(\Theta_{p_{i}p_{j}})(\partial^{\alpha}\partial_{x_{j}}w^{n})\ dxd\tau,
\end{equation}
\begin{equation}\nonumber
\widetilde{IV}_{A,ii}=\varepsilon\int_{0}^{t}\int_{\mathbb{T}^{d}}(\partial^{\alpha}\mu^{n+1})\sum_{i=1}^{d}\sum_{j=1}^{d}
\left(\partial_{x_{i}}\left((\mu^{n}+\bar{m})\Theta_{p_{i}p_{j}}\right)\right)(\partial^{\alpha}\partial_{x_{j}}w^{n})\ dxd\tau.
\end{equation}
The important point here is that the term $\widetilde{IV}_{A,i}$ involves up to $s+1$ derivatives of $\mu^{n+1}$ as well as 
up to $s+1$ derivatives of $w^{n}.$  Therefore neither of these factors are controlled by $\widetilde{M}_{n+1}$ or $\widetilde{M}_{n},$ and
both must be contolled instead by $\widetilde{N}_{n+1}$ or $\widetilde{N}_{n}.$  In the proof of Theorem \ref{existenceTheorem}, we did
not have a product of two such terms which needed to be controlled by $N_{n+1}$ or $N_{n}.$  

Of course, in addition to $\widetilde{IV}_{A},$ we also have terms $\widetilde{I},$ $\widetilde{II},$ $\widetilde{III},$ and $\widetilde{IV}_{B};$ 
all of these, as well as $\widetilde{IV}_{A,ii},$ can be estimated just as in the proof of Theorem \ref{existenceTheorem}, yielding the
following estimate:
\begin{multline}\label{allButIVai}
\frac{1}{2}\int_{\mathbb{T}^{d}}(\partial^{\alpha}\mu^{n+1}(t,x))^{2}\ dx
-\frac{1}{2}\int_{\mathbb{T}^{d}}(\partial^{\alpha}\mu^{n+1}(0,x))^{2}\ dx
+\int_{0}^{t}\int_{\mathbb{T}^{d}}|D\partial^{\alpha}\mu^{n+1}|^{2}\ dxd\tau
\\
\leq \widetilde{IV}_{A,i}+
\frac{1}{8}\sup_{t\in[0,T]}\|\partial^{\alpha}\mu^{n+1}\|_{0}^{2}
+c\varepsilon T(F(M_{n}))^{2}\left((1+T)(1+N_{n})(1+M_{n})^{s+2}\right).
\end{multline}

Next we estimate the important term, $\widetilde{IV}_{A,i}.$  We use Young's inequality without a small parameter,
and we also use \eqref{newSmall}, finding
\begin{multline}\label{ivaidone}
\widetilde{IV}_{A,i}\leq \frac{\varepsilon\varepsilon_{1}d}{2}(1+\widetilde{M}_{n})^{1/2}
\sum_{i=1}^{d}\int_{0}^{t}\int_{\mathbb{T}^{d}}|\partial^{\alpha}\partial_{x_{i}}\mu^{n+1}|^{2}\ dxd\tau
\\
+
\frac{\varepsilon\varepsilon_{1}d^{2}}{2}(1+\widetilde{M}_{n})^{1/2}\widetilde{N}_{n}.
\end{multline}

By taking the product $\varepsilon\varepsilon_{1}$ small enough, and combining \eqref{ivaidone} into \eqref{allButIVai}, we find
\begin{multline}\nonumber
\frac{1}{2}\int_{\mathbb{T}^{d}}(\partial^{\alpha}\mu^{n+1}(t,x))^{2}\ dx
-\frac{1}{2}\int_{\mathbb{T}^{d}}(\partial^{\alpha}\mu^{n+1}(0,x))^{2}\ dx
\\
+\frac{1}{2}\int_{0}^{t}\int_{\mathbb{T}^{d}}|D\partial^{\alpha}\mu^{n+1}|^{2}\ dxd\tau
\leq
\frac{1}{8}\sup_{t\in[0,T]}\|\partial^{\alpha}\mu^{n+1}\|_{0}^{2}
\\
+c\varepsilon\left(\varepsilon_{1}+ T(F(M_{n}))^{2}\right)\left((1+T)(1+N_{n})(1+M_{n})^{s+2}\right).
\end{multline}
The rest of the proof then follows as in Theorem \ref{existenceTheorem}.
\end{proof}

Finally, we close this section by remarking that we could also prove uniqueness in this case (i.e., the case of a general, non-regularizing
payoff function), just as we proved Theorem \ref{uniquenessTheorem} above.

\subsection{Extension to problems with congestion}

Some other works have considered mean field games with congestion, in which the Hamiltonian has the particular form
with a power of $m$ in the denominator.  The solutions proven to exist in 
\cite{gomesCongestion}, for example, are shown to always satisfy $m>0,$ then.  Our analysis could be extended to 
such Hamiltonians
when the data satisfies $m_{0}>0.$  
In this case, smallness constraints can again be used to keep $m$ positive and to control the
nonlinear evolution.  We now make this precise.

\begin{corollary}\label{congestionTheorem} Let $s>\left\lceil\frac{d+5}{2}\right\rceil$ be given, and  let $\gamma>0$ be given.
Let $u_{T}\in H^{s}$ and $m_{0}\in H^{s-1}$ be given, such that $m_{0}$ satisfies $m_{0}(x)>\gamma$ for all
$x\in\mathbb{T}^{d}.$
Let the Hamiltonian take the form 
\begin{equation}
\mathcal{H}(t,x,m,Du)=\phi(m)H_{1}(t,x,m,Du)+H_{2}(t,x,m,Du),
\end{equation} 
where $H_{1}$ and $H_{2}$ each satisfy
{\bf(H1)} and where $\phi:(0,\infty)\rightarrow\mathbb{R}$ is smooth and satisfies $|\phi(q)|\rightarrow\infty$
as $q\rightarrow0.$  Let $F$ be as above.  If the product $\varepsilon T F$ is sufficiently small,
the system \eqref{uEquation}, \eqref{mEquation}, \eqref{planningBC} has a classical solution
$(u.m)\in L^{\infty}([0,T];H^{s}\times H^{s-1}).$
\end{corollary}

So, for example, we could let $H_{1}$ and $H_{2}$
be given by the examples discussed in Remark \ref{existenceRemark} or Remark \ref{uniquenessRemark}, for instance, and we could let
$\phi(m)$ be given by $\ln(m)$ or $m^{\sigma}$ for any $\sigma<0.$  The proof of Corollary \ref{congestionTheorem}
is a brief application of our main theorem and further application of our smallness conditions; we now provide this
proof.

\begin{proof}
Let $\eta\in(0,1)$ be given.  We want to maintain the property that $m(t,x)>\eta\gamma$ for all $x$ and $t.$
We introduce a cutoff function $\psi,$ such that $\psi:\mathbb{R}\rightarrow\mathbb{R}$ is smooth and monotone,
such that $\psi(\beta)=\eta\gamma/4$ for $\beta\leq\eta\gamma/4,$ and $\psi(\beta)=\beta$ for $\beta\geq\eta\gamma/2.$
We introduce a modified Hamiltonian, $\bar{H},$ defined by
\begin{equation}\nonumber
\bar{H}(t,x,m,Du)=\phi(\psi(m))H_{1}(t,x,m,Du)+H_{2}(t,x,m,Du).
\end{equation}
Then the modified Hamiltonian $\bar{H}$ satisfies {\bf{(H1)}}, and thus Theorem \ref{existenceTheorem} applies.  Thus, for 
$\varepsilon TF$ sufficiently small, there exists a solution to the problem
\begin{equation}\nonumber
u_{t}+\Delta u+\varepsilon \bar{H}(t,x,m,Du)=0,
\end{equation}
\begin{equation}\label{modifiedMEquation}
m_{t}-\Delta m + \varepsilon\mathrm{div}(m\bar{H}_{p}(t,x,m,Du))=0,
\end{equation}
with data given by \eqref{planningBC}.

We now consider \eqref{modifiedMEquation} in more detail; we write the Duhamel
formula for this:
\begin{equation}\nonumber
m(t,\cdot)=e^{-\Delta t}m_{0}-\varepsilon\int_{0}^{t}e^{-\Delta(t-s)}\mathrm{div}(m\bar{H}_{p}(s,\cdot,m, Du)\ ds.
\end{equation}
We know from the maximum principle that $(e^{\Delta t}m_{0})(t,x)>\gamma$ for all $t$ and $x.$
To show that $m(t,x)>\eta\gamma$ for all $x$ and $t,$ it is sufficient to demonstrate
\begin{equation}\nonumber
\left|\varepsilon\int_{0}^{t}e^{-\Delta(t-s)}\mathrm{div}(m\bar{H}_{p}(t,\cdot,m,Du)\ ds\right|_{\infty}\leq (1-\eta)\gamma,
\end{equation}
for all $t.$  For $k>d/2,$ by Sobolev embedding and the triangle inequality we have
\begin{multline}\nonumber
\left|\varepsilon\int_{0}^{t}e^{-\Delta(t-s)}\mathrm{div}(m\mathcal{H}_{p}(t,\cdot,m,Du))\ ds\right|_{\infty}\\
\leq\left\|\varepsilon\int_{0}^{t}e^{-\Delta(t-s)}\mathrm{div}(m\mathcal{H}_{p}(t,\cdot,m,Du))\ ds\right\|_{H^{k}}
\\
\leq\varepsilon\int_{0}^{t}\|e^{-\Delta(t-s)}\mathrm{div}(m\mathcal{H}_{p}(t,\cdot,m,Du))\|_{H^{k}}\ ds.
\end{multline}
Then using the action of the heat semigroup on Sobolev spaces and making an elementary supremum inequality, we 
have
\begin{multline}\nonumber
\left|\varepsilon\int_{0}^{t}e^{-\Delta(t-s)}\mathrm{div}(m\mathcal{H}_{p}(t,\cdot,m,Du))\ ds\right|_{\infty}
\\
\leq \varepsilon\int_{0}^{t}\|\mathrm{div}(m\mathcal{H}_{p}(t,\cdot,m,Du))\|_{H^{k}}\ ds
\leq\varepsilon T \sup_{t\in[0,T]}\|m\mathcal{H}_{p}(t,\cdot,m,Du)\|_{H^{k+1}}.
\end{multline}
So, as long as
\begin{equation}\label{conditionCongestion}
\varepsilon T \sup_{t\in[0,T]}\|m\mathcal{H}_{p}(t,\cdot,m,Du)\|_{H^{k+1}} \leq (1-\eta)\gamma,
\end{equation}
we know that $m(t,x)>\eta\gamma$ for all $t$ and $x.$  To be definite, we take $k$ to be the integer ceiling of $(d+1)/2.$

We know that the solutions guaranteed to exist by Theorem \ref{existenceTheorem} satisfy 
\begin{equation}\nonumber
\sup_{t\in[0,T]}\|Du\|_{s-1}^{2}+\|\mu\|_{s-1}^{2}\leq 8\left(\|Du_{T}\|_{s-1}^{2}+\|\mu_{0}\|_{s-1}^{2}\right).
\end{equation}
By Lemma \ref{boundsOnTheta}, then, we see that there exists a constant $\bar{K}>0$ such that 
\begin{equation}\nonumber
\sup_{t\in[0,T]}\|mH_{p}(t,\cdot,m,Du)\|_{H^{k+1}}\leq \bar{K}\left(\varepsilon T
F(8\|Du_{T}\|_{s-1}^{2}+8\|\mu_{0}\|_{s-1}^{2})\right).
\end{equation}
Thus, if the product $\varepsilon T F$ is sufficiently small, we do have that condition \eqref{conditionCongestion}
is satisfied.  Therefore $m(t,x)>\eta\gamma$ for all $t$ and $x,$ and thus $\psi(m)=m.$  Since $\psi(m) =m,$
we have $\tilde{H}(t,x,m,Du)=\mathcal{H}(t,x,m,Du)$ for our solutions $(u,m).$  This completes the proof.
\end{proof}

\subsection{Final remarks}

Future work to be done includes studying questions regarding regularity and uniqueness.
For questions of regularity, this includes both lowering regularity requirements on the data to get
results such as those in the present work, and also inferring still higher regularity of solutions
such as those we have proved to exist.  The uniqueness question also certainly deserves further
attention; while we have presented a uniqueness theorem with a smallness constraint, as have
Bardi and Cirant \cite{bardiCirant}, as we have discussed in the introduction, Bardi and Fischer and Cirant and Tonon
have given examples of non-uniqueness \cite{bardiFischer}, \cite{cirantTonon}.  Understanding when solutions are 
and are not unique and characterizing the mutiple possible solutions is an important problem to
be studied.

\section*{Acknowledgments}
The author gratefully acknowledges support from the National Science Foundation through 
grants  DMS-1515849 and DMS-1907684.  The author is also grateful to the Institute for Computational and Experimental Research
in Mathematics (ICERM) at Brown University, for hosting him while this paper was substantially completed.
The author also is grateful to the Institute for Pure and Applied Mathematics (IPAM) at University of California, Los Angeles
for hosting him while the article was finished.
\bibliography{ambroseMeanField}{}
\bibliographystyle{plain}

\end{document}